\documentclass[10pt,twocolumn,amsmath,amssymb,aps,pra,secnumarabic,
    nofootinbib,groupedaddress]{revtex4-1}

\usepackage{mathtools,mathptmx,amsthm,tikz,tikz-cd,comment,enumitem}
\usepackage[colorlinks=true,linkcolor=red!70!black,citecolor=green!70!black,
    urlcolor=magenta!70!black,backref]{hyperref}

\usetikzlibrary{calc,decorations.markings,decorations.pathreplacing,arrows}

\tikzset{mybrace/.style={decorate,decoration={brace,amplitude=10pt,aspect=#1}}}

\setlist[enumerate,1]{font=\bfseries,label=\arabic*.}

\makeatletter\def\@bibdataout@init{}\def\pre@bibdata{}\makeatother
\colorlet{darkgreen}{green!50!black}
\colorlet{darkblue}{blue!70!black}
\colorlet{darkred}{red!70!black}

\newtheorem{theorem}{Theorem}[section]
\newtheorem{proposition}[theorem]{Proposition}
\newtheorem{lemma}[theorem]{Lemma}
\newtheorem{corollary}[theorem]{Corollary}
\theoremstyle{definition}

\theoremstyle{remark}
\newtheorem*{remark}{Remark}

\newcommand{\eg}{\emph{e.g.}}
\newcommand{\ie}{\emph{i.e.}}
\newcommand{\Ie}{\emph{I.e.}}

\newcommand{\FP}{\mathsf{FP}}

\newcommand{\shP}{\mathsf{\#P}}

\newcommand{\NP}{\mathsf{NP}}

\newcommand{\CSAT}{\mathsf{CSAT}}
\newcommand{\shCSAT}{\mathsf{\#CSAT}}

\newcommand{\ZSAT}{\mathsf{ZSAT}}

\newcommand{\shRSAT}{\mathsf{\#RSAT}}
\newcommand{\shZSAT}{\mathsf{\#ZSAT}}

\newcommand{\Alt}{\operatorname{Alt}}
\newcommand{\Aut}{\operatorname{Aut}}

\newcommand{\Inn}{\operatorname{Inn}}

\newcommand{\Rub}{\operatorname{Rub}}

\newcommand{\Sym}{\operatorname{Sym}}

\newcommand{\ab}{{\operatorname{ab}}}
\newcommand{\ev}{\operatorname{ev}}
\newcommand{\free}{{\operatorname{free}}}

\newcommand{\sch}{\operatorname{sch}}
\newcommand{\tor}{{\operatorname{tor}}}
\renewcommand{\wr}{\operatorname{wr}}

\newcommand{\onto}{\twoheadrightarrow}
\newcommand{\into}{\hookrightarrow}
\newcommand{\longto}{\longrightarrow}
\renewcommand{\setminus}{\smallsetminus}
\newcommand{\normaleq}{\unlhd}

\newcommand{\yes}{\mathrm{yes}}
\newcommand{\no}{\mathrm{no}}

\newcommand{\Z}{\mathbb{Z}}
\newcommand{\R}{\mathbb{R}}
\newcommand{\N}{\mathbb{N}}

\newcommand{\hG}{\hat{G}}
\newcommand{\hR}{\hat{R}}
\newcommand{\hU}{{\hat{U}}}

\newcommand{\tC}{\tilde{C}}
\newcommand{\tG}{\tilde{G}}

\newcommand{\Cor}[1]{Corollary~\ref{#1}}

\newcommand{\Fig}[1]{Figure~\ref{#1}}
\newcommand{\Lem}[1]{Lemma~\ref{#1}}
\newcommand{\Prop}[1]{Proposition~\ref{#1}}
\newcommand{\Sec}[1]{Section~\ref{#1}}
\newcommand{\Thm}[1]{Theorem~\ref{#1}}

\newcommand{\defeq}{\stackrel{\mathrm{def}}=}

\newenvironment{eq}[1]{\begin{equation} \label{#1}}
    {\end{equation}\ignorespacesafterend}

\begin{document}
\title{Coloring invariants of knots and links are often intractable}
\author{Greg Kuperberg}
\email{greg@math.ucdavis.edu}
\thanks{Partly supported by NSF grant CCF-1716990}
\affiliation{University of California, Davis}

\author{Eric Samperton}
\email{eric@math.ucsb.edu}
\thanks{Partly supported by NSF grant CCF-1716990}
\affiliation{University of California, Santa Barbara}

\date{\today}

\begin{abstract} Let $G$ be a nonabelian, simple group with a nontrivial
conjugacy class $C \subseteq G$.  Let $K$ be a diagram of an oriented knot
in $S^3$, thought of as computational input.  We show that for each such $G$
and $C$, the problem of counting homomorphisms $\pi_1(S^3\setminus K) \to G$
that send meridians of $K$ to $C$ is almost parsimoniously $\shP$-complete.
This work is a sequel to a previous result by the authors that counting
homomorphisms from fundamental groups of integer homology 3-spheres to
$G$ is almost parsimoniously $\shP$-complete.  Where we previously used
mapping class groups actions on closed, unmarked surfaces, we now use
braid group actions.
\end{abstract}

\maketitle

\section{Introduction}
\label{s:intro}

Let $K$ be an oriented knot in the 3-sphere $S^3$ described by some given
knot diagram.  Fox \cite[Exer.~VI.6-7]{CF:knot} popularized the idea of
a 3-coloring of the diagram $K$, which is now also called a Fox coloring
\cite{dHJ:graph}.  By definition, such a coloring is an assignment of one
of three colors to each arc in $K$ such that at every crossing, the over-arc
and the two other arcs are either all the same color or all different colors.
It is easy to check that the number of 3-colorings of a diagram is invariant
under Reidemeister moves, and is therefore an isotopy invariant of $K$.

Fox colorings are a special case of the following type of generalized
coloring based on the Wirtinger presentation of the knot group $\pi_1(S^3
\setminus K)$ \cite{Reidemeister:knoten}: Fix a finite group $G$
and a conjugacy class $C \subseteq G$ that generates $G$.  Then a
\emph{$C$-coloring} is an assignment of an element $c \in C$ to each
arc in $K$ such that at each crossing, one of the two relations as in
\Fig{f:relations} holds, depending on the sign of the crossing.  The set
of $C$-colorings is bijective with the set
\[ H(K;G,C) \defeq \{ f:\pi_1(S^3 \setminus K)
    \to G \mid f(\gamma) \in C \}\]
of homomorphisms from the knot group to $G$ that take a meridian $\gamma$
of $K$ to some element in $C$.  (Since the meridians are themselves a
conjugacy class of $\pi_1(S^3 \setminus K)$, it doesn't matter which one
we choose.)  Then $\#H(K;G,C)=|H(K;G,C)|$ is an important integer-valued
invariant of knots.

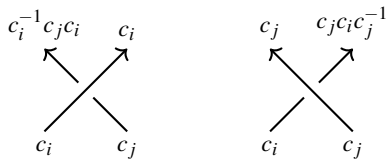
\begin{figure}[h]
\begin{tikzpicture}[thick]
\draw[->] (1.1,0) -- (0,1.1);
\draw[->, white,line width=8] (0,0) -- (1.1,1.1);
\draw[->] (0,0) -- (1.1,1.1);
\draw (0,0) node[below] {$c_i$};
\draw (1.1,0) node[below] {$c_j$};
\draw (1.1,1.1) node[above] {$c_i$};
\draw (0,1.1) node[above] {$c_i^{-1}c_jc_i$};
\begin{scope}[xshift=3cm]
\draw[->] (0,0) -- (1.1,1.1);
\draw[->, white,line width=8] (1.1,0) -- (0,1.1);
\draw[->] (1.1,0) -- (0,1.1);
\draw (0,0) node[below] {$c_i$};
\draw (1.1,0) node[below] {$c_j$};
\draw (1.1,1.1) node[above] {$c_jc_ic_j^{-1}$};
\draw (0,1.1) node[above] {$c_j$};
\end{scope} \end{tikzpicture} \hfill
\caption{The Wirtinger relations.}
\label{f:relations} \end{figure}

If $G = S_3$ is the symmetric group on 3-letters and $C$ is the conjugacy
class of transpositions, then $H(K;G,C)$ is precisely the set of Fox
colorings of $K$.  In this case, and in any case when $G$ is metabelian,
$H(K;G,C)$ is an abelian group (or more precisely a torsor over one)
that can be calculated efficiently using the Alexander polynomial of $K$.
However, Fox also considered the set $H(K;G,C)$ for general $G$ and $C$.
When $G = A_5$, he observed that ``$A_5$ is a simple group, so that I
know of no method of finding representations on $A_5$ other than just
trying" \cite{Fox:quick}.  Our main result, \Thm{th:main}, demonstrates
that Fox's frustration was prescient; see \Sec{ss:interp}.

To state our precise result, we first refine the invariants $H(K;G,C)$
and $\#H(K;G,C)$.  Let $\Aut(G,C)$ be the group of automorphisms of $G$
that take $C$ to itself.  Then $\Aut(G,C)$ acts on $\#H(K;G,C)$,
and in particular it acts freely on the surjective maps in $H(K;G,C)$.
Let
\[ Q(K;G,C) \defeq \{ f:\pi_1(S^3 \setminus K)
    \onto G \mid f(\gamma) \in C \}/\Aut(G,C) \]
be the corresponding quotient set.  Regardless of $K$, the set $H(K;G,C)$
always contains a unique homomorphism with cyclic image that sends $\gamma$
to each given $c \in C$.  If $G$ is not cyclic and if all other homomorphisms
are surjective, then in these cases
\begin{eq}{e:almost} \#H(K;G,C) = \#C + \#\Aut(G,C) \cdot \#Q(K;G,C). \end{eq}
Our main theorem implies that if $G$ is non-abelian simple, then $\#Q(K;G,C)
= \#Q(K,\gamma;G,c)$ is computationally intractable, and remains so even when
every homomorphism $f \in H(K;G,C)$ is promised either to be surjective
or have cyclic image.

\begin{theorem} Let $G$ be a fixed, finite, non-abelian simple group,
and fix a nontrivial conjugacy class $C \subseteq G$. If $K \subseteq S^3$
is an oriented knot specified by a knot diagram interpreted as
computational input, then the invariant $\#Q(K;G,C)$ is parsimoniously
$\shP$-complete.  The reduction also guarantees that $\#Q(K;J,E) = 0$
for any group $J$ generated by a conjugacy class $E$ with $\#E < \#C$,
except when $J$ is a cyclic group.
\label{th:main} \end{theorem}

We note that $J$ in the statement of the theorem is not necessarily a
subgroup of $G$, although the case that $J$ is a subgroup of $G$ generated
by a subset of $C$ is of particular interest.

Before reviewing the definition of $\shP$-completeness and interpreting
\Thm{th:main}, we expand on the relation between $H(K;G,C)$ and $Q(K;G,C)$.

Let $c \in C$ and let
\[ H(K,\gamma;G,c) \defeq \{ f:\pi_1(S^3 \setminus K) \to G \mid
    f(\gamma) = c \}. \]
It is easy to see (by conjugation in $G$) that $\#H(K,\gamma;G,c)$
is independent of the choice of $c$ and that
\[ \#H(K;G,C) = \#C \cdot \#H(K,\gamma;G,c). \]
Let $\Aut(G,c)$ be the group of automorphisms of $G$ that fix $c$.
Then $\Aut(G,c)$ acts on $H(K,\gamma;G,c)$, and in particular it acts
freely on the surjective maps in $H(K,\gamma;G,c)$.  Let
\[ Q(K,\gamma;G,c) \defeq \{ f:\pi_1(S^3 \setminus K) \onto G \mid
    f(\gamma) = c \}/\Aut(G,c) \]
be the corresponding quotient set.  Again by examining conjugation in $G$,
we learn that the natural map $Q(K,\gamma;G,c)$ to $Q(K;G,C)$ is a bijection.

Given that every $f \in H(K;G,c)$
has some image $J \ni c$, we obtain the summation formula
\[ \#H(K,\gamma;G,c) = \sum_{c \in J \leq G} \#\Aut(J,c)
    \cdot \#Q(K,\gamma;J,c).\]
Given that the conjugacy class of $\gamma$ generates $\pi(S^3 \setminus
K)$, the conjugacy class $E$ of $c$ in $J$ generates $J$ as well.  So we
can also write
\begin{eq}{e:sum} \#H(K;G,C) = \sum_{c \in E \subseteq J \leq G}
    \#C \cdot \#\Aut(J,E) \cdot \#Q(K;J,E).\end{eq}
Finally if $J \ne G$, then necessarily $\#E < \#C$.  Thus when the
conclusion of \Thm{th:main} holds, equation \eqref{e:sum} reduces to
equation \eqref{e:almost}.

\subsection{Interpretation and previous results}
\label{ss:interp}

For an introduction to the topic of computational complexity, see our
previous article \cite[Sec.~2.1]{K:zombies}, as well as Arora-Barak
\cite{AB:modern} and the Complexity Zoo \cite{W:zoo}.  Here we just give
a brief description the concept of $\shP$-completeness and parsimonious
reduction.

If $A$ is a finite alphabet and $A^*$ is the set of finite words in $A$,
a problem in $\shP$ is by definition a function $c:A^* \to \N$ given
by the equation
\[ c(x) = \#\{y \mid p(x,y) = \yes \}, \]
where length of the certificate $y$ is polynomial in the length of $x$, and
\[ p:A^* \times A^* \to \{\yes,\no\} \]
is a predicate that can be computed in polynomial time.  A counting
problem $c \in \shP$ is \emph{parsimoniously $\shP$-complete} when every
problem $b \in \shP$ can be converted to a special case of $c$. More
precisely, $c$ is parsimoniously $\shP$-complete when $b(x) = c(f(x))$
for some function $f:A^* \to A^*$ that can be computed in polynomial time.

The significance of parsimonious $\shP$-completeness for a counting problem
$c$ is that not only is the exact value of $c$ computationally intractable,
but also that obtaining any partial information about $c$ is computationally
intractable, assuming standard conjectures in complexity theory.  To give
a contrasting example, the number of perfect matchings $m(\Gamma)$ of a
finite, bipartite graph $\Gamma$ is well-known to be $\shP$-complete by
the looser standard of Turing-Cook reduction \cite{Valiant:permanent}.
The exact value of $m(\Gamma)$ is thus intractable.  However, the
parity of $m$ can be computed in polynomial time (as a determinant
over $\Z/2$), whether $m(\Gamma) = 0$ can be computed in polynomial time
\cite{Munkres:assignment}, and $m(\Gamma)$ can be approximated in randomized
polynomial time \cite{JSV:approx}.  Barring a catastrophe in computer
science, no such partial results are possible for computing $\#Q(K;G,C)$
under the hypotheses of \Thm{th:main}, not even with the aid of a quantum
computer \cite{BBBV:strengths}.

The analogous concepts for existence questions are the complexity class
$\NP$ and the $\NP$-completeness property.   A decision function $d:A^*
\to \{\yes,\no\}$ is in $\NP$ if there is a polynomial-time predicate $p$
such that $d(x) = \yes$ if and only if $p(x,y) = \yes$.  The function $d$
is \emph{Post-Karp $\NP$-complete} if for every $e \in \NP$, $e(x) = d(f(x))$
for some $f$ computable in polynomial time.

In particular, \Thm{th:main} implies $\NP$-completeness results for the
existence of $C$-colorings.  De Mesmay, Rieck, Sedgiwck, and Tancer
\cite{dMRST:unbearable} cite us for the first known $\NP$-hardness
result for knots in $S^3$ (as opposed to knots in general $3$-manifolds
\cite{AHT:genus}), so we record this as an explicit corollary.

\begin{corollary} Let $G$ be a fixed, finite, non-abelian simple group,
and fix a nontrivial conjugacy class $C \subseteq G$.  If $K \subseteq S^3$
is an oriented knot specified by a knot diagram interpreted as computational
input, then deciding whether $Q(K;G,C)$ is non-empty or $\#H(K,G,C) > \#C$
is $\NP$-complete via Post-Karp reduction.
\label{c:main} \end{corollary}

Note also that much more is true thanks to a result of Valiant and
Vazirani \cite{VV:unique}: Distinguishing \emph{any} two values of a
parsimoniously $\shP$-hard problem is $\NP$-hard with randomized reduction
\cite[Thm.~2.1]{K:zombies}.

Some partial information about the unadjusted counting invariant $\#H(K;G,C)$
can be computed efficiently; for instance, that it always at least $\#C$.
However, \Thm{th:main} and equation~\eqref{e:almost} together imply that
this extra information can be trivial.  We call a counting problem $c
\in \shP$ \emph{almost parsimoniously $\shP$-complete} if for every $b
\in \shP$, there is a reduction $\alpha b(x) + \beta = c(f(x)))$ for some
universal constants $\alpha > 0$ and $\beta \ge 0$.  Almost parsimonious
reductions arise naturally in computational complexity.  For example, the
number of 3-colorings of a planar graph with at least one edge is always
divisible by 6; but after dividing by 6, this number becomes parsimoniously
$\shP$-complete \cite{Barbanchon:unique}.  Likewise, \Thm{th:main}
shows $\#H(K;G,C)$ is almost parsimoniously $\shP$-complete.

The strongest partial result toward \Thm{th:main} to our knowledge is
that of Krovi and Russell.  Taking the straightforward generalization of
$H(K;G,C)$ to links $L$, they showed that $\#H(L;A_m,C)$ is $\shP$-complete
for any fixed $m \ge 5$ and any fixed conjugacy class $C$ of permutations
with at least 4 fixed points \cite{KR:finite}.  Their reduction is not
almost parsimonious because it has an error term.  In particular, they do
not obtain that it is $\NP$-complete to determine whether $\#H(L;A_m,C)
> \#C$ or $\#Q(L;A_m,C) > 0$.

\subsection{Outline of the proof}
\label{ss:outline}

Our proof of \Thm{th:main} follows our proof of the analogous theorem for
homology 3-spheres \cite{K:zombies}, which we assume as a prerequisite for
this article.  However, \Thm{th:main} is a stronger result because knots
are a more restricted class of topological objects.  As a preliminary
observation, both $\#H(K;G,C)$ and $\#Q(K;G,C)$ are in $\shP$ by the same
argument as in the 3-manifold case \cite[Thm.~2.7]{K:zombies}.

The reduction begins with a counting version of circuit satisfiability,
$\shCSAT$, that is rather directly parsimoniously $\shP$-complete
\cite[Thm.~2.2]{K:zombies}.  The $\shCSAT$ problem can be reduced to a
certain version with reversible circuits, $\shRSAT$, and we can assume in
both problems that circuits are planar.  Whereas the output to a $\CSAT$
circuit is constrained to $\yes$ and the input is any satisfying certificate,
both the input and output of a $\shRSAT$ circuit are partially constrained.
In turn, $\shRSAT$ reduces almost parsimoniously to an ad hoc reversible
circuit problem called $\shZSAT$ where:  (1) The alphabet is a $U$-set
for some finite group $U$ with a single fixed point called the ``zombie"
symbol and otherwise free orbits, and (2) the gates are $U$-equivariant
permutations.  Finally $\shZSAT$ reduces to $\#Q(K;G,C)$ in a construction in
which the circuit becomes a braid word and suitable initialization and
finalization conditions are expressed by plat closure.

Let $D^2 \setminus [n]$ denote a disk with $n$ punctures.  The reduction
from $\shZSAT$ to $\#Q(K;G,C)$ involves a braid group action on the
set of surjections
\[ f:\pi_1(D^2 \setminus [2k]) \onto G \]
with clockwise monodromy in $C$ at $k$ punctures, counterclockwise monodromy
in $C$ at the other $k$ punctures, and trivial monodromy on the outside.
In \Thm{th:refine}, we show that when $k$ is large enough, this braid
group action is very highly transitive modulo a certain Schur invariant.
High transitivity makes it possible to implement gates in a precise way that
preserves enumeration and does not disturb non-surjective homomorphisms.
\Thm{th:refine} in turn requires two types of group-theoretic ingredients.
The first ingredient, \Thm{th:evw}, is a refinement of the Conway-Parker
theorem \cite{FV:inverse} that shows that the action is at least
transitive when $G$ is any finite group.  This refinement first appeared in
version 1 of a retracted e-print of Ellenberg, Venkatesh, and Westerland
\cite[Thm.~7.6.1]{EVW:hurwitz2}; the second author of this article later
found a topological proof \cite[Thm.~1.1]{Samperton:schur}.  The second
ingredient is a set of surjectivity results for group homomorphisms
(\Sec{s:group}).

\section{Group theory}
\label{s:group}

In this section, we simply list some surjectivity results in group theory
that we will need to prove \Thm{th:refine}.

\subsection{Surjectivity for products}

The following lemma is a mutual corollary of Goursat's Lemma
\cite{Goursat:divisions} and Ribet's Lemma \cite{Ribet:division,Ribet:adic}.
In our research, we first saw it stated by Dunfield and Thurston
\cite[Lem.~3.7]{DT:random}.

\begin{lemma}[After Goursat-Ribet
    {\cite[Lems.~5.2.1~\&~5.2.2]{Ribet:division}}] If
\[ f:J \to G_1 \times G_2 \times \dots \times G_n \]
is a homomorphism from a group $J$ to a product of non-abelian simple groups
that surjects onto each factor, and if no two factor homomorphisms $f_i:
J \onto G_i$ and $f_j:J \onto G_j$ are equivalent by an isomorphism $G_i
\cong G_j$, then $f$ is surjective.
\label{l:goursrib} \end{lemma}

\begin{remark} Results similar to \Lem{l:goursrib} have appeared many times
in the literature with various attributions and extra hypotheses.  Both Ribet
and Dunfield-Thurston assume that the target groups are finite even though
their proofs do not use this hypothesis.  Dunfield-Thurston also state that
the result is due to Hall \cite{Hall:eulerian}.   However, all that we can
find in this citation is an unproven lemma (in his Section 1.6) that can
(with a little work) be restated as a special case of \Lem{l:goursrib}.
\end{remark}

Say that group $G$ is \emph{Zornian} if every proper normal subgroup of
$G$ is contained in a maximal normal subgroup.   Clearly every finite
group is Zornian, which is the case that we will need; more generally
every finitely generated group is Zornian \cite[Sec.~3.2]{K:zombies}.
The following is also an adaptation of Goursat's Lemma.

\begin{lemma}[After Goursat {\cite[Lem.~3.6]{K:zombies}}] Suppose that
\[ f:J \to G_1 \times G_2 \]
is a group homomorphism that surjects onto $G_1$, and suppose that $G_1$
is Zornian.  If no simple quotient of $G_1$ is involved in $G_2$, then
$f(B)$ contains $G_1$.
\label{l:zorn} \end{lemma}

Finally we will need the following two related lemmas.

\begin{lemma}[Ribet {\cite[Sec.~5.2]{Ribet:division}}] If
\[ N \normaleq G = G_1 \times G_2 \times \cdots \times G_n \]
is a normal subgroup of a product of perfect groups that
surjects onto each factor $G_i$, then $N = G$.
\label{l:ribetn} \end{lemma}

\begin{lemma}[{\cite[Lem.~3.7]{K:zombies}}] If
\[ f: G_1 \times G_2 \times \cdots \times G_n \onto J \]
is a surjective homomorphism from a direct product of groups to a nonabelian
simple quotient $J$, then it factors through a quotient map $f_i :G_i
\onto J$ for a single value of $i$.
\label{l:prodquo} \end{lemma}

\subsection{Rubik groups}

Let $G$ be a group and let $X$ be a $G$-set with finitely many orbits.
We denote the group of $G$-set automorphisms of $X$ by $\Sym_G(X)$.
We define the \emph{Rubik group} to be the commutator subgroup
\[ \Rub_G(X) \defeq [\Sym_G(X), \Sym_G(X)]. \]
Note that the natural map $\Sym_G(X) \to \Sym(X/G)$ takes $\Rub_G(X)$
to $\Alt(X/G)$.

When $X$ is a free $G$-set with $\#(G/X) = n$, $\Sym_G(X)$ is isomorphic to
the restricted wreath product
\[ \Sym(n,G) \defeq G \wr_m \Sym(n) = G^{\times n} \rtimes \Sym(n). \]
Likewise let
\[ \Rub(n,G) \defeq [\Sym(n,G),\Sym(n,G)].\]
We need two results about Rubik groups from our previous work
\cite{K:zombies} which we will restate here.

The first result is an elementary counterpart for Rubik groups to the
well-known corollary of the classification of finite simple groups that a
6-transitive subgroup of $\Sym(n)$ is \emph{ultratransitive}, by definition
that it contains $\Alt(n)$ or equivalently that it is $(n-2)$-transitive.
Say that a group homomorphism
\[ f:J \to \Sym(n,G)\]
is \emph{$G$-set $k$-transitive} if it acts transitively on ordered lists
of $k$ elements that all lie in distinct $G$-orbits.  Say likewise that
it is \emph{$G$-set ultratransitive} if its image contains $\Rub(n,G)$.

\begin{theorem}[{\cite[Thm.~3.10]{K:zombies}}]
Let $G$ be a group and let $n \ge 7$ be an integer such that $\Alt(n-2)$
is not a quotient of $G$.  Suppose that a homomorphism
\[ f:J \to \Sym(n,G) \]
from a group $J$ is $G$-set 2-transitive and that its
projection $\Rub(n,G) \to \Alt(n)$ is 6-transitive (and therefore
ultratransitive).  Then $f$ is $G$-set ultratransitive.
\label{th:rubik} \end{theorem}

The other result says $\Rub(n,G)$ has a unique simple quotient when $\Alt(n)$
is a simple group.

\begin{lemma}[{\cite[Lem.~3.11]{K:zombies}}] If $G$ is any group and $n \ge
5$, then the only simple quotient of $\Rub(n,G)$ is $\Alt(n)$.
\label{l:only} \end{lemma}

\section{Equivariant circuits}
\label{s:circuits}

In this section, we review the $\ZSAT$ circuit model from our previous
work \cite{K:zombies}.

Let $A$ be a finite set with at least two elements, considered as a
computational alphabet.  A \emph{reversible circuit of width $n$} is a
bijection $A^n \to A^n$ expressed as a composition of bijective gates $A^k
\to A^k$ in the pattern of a directed, acyclic graph.  The gates are all
chosen from some fixed finite set of bijections.

Let $G$ be a non-trivial finite group acting on $A$ with a single fixed
point $z$, the \emph{zombie symbol}, and otherwise with free orbits.
Let $I$ and $F$ be two proper $G$-invariant subsets of $A$ that
contain the zombie symbol and are not just that symbol:
\[ \{z\} \subsetneq I, F \subsetneq A. \]
We interpret $I$ as an initial subalphabet and $F$ as a final subalphabet.
An instance $Z$ of $\ZSAT_{G,A,I,F}$ is a planar reversible circuit with gate
set $\Rub_G(A^2)$.  (Remark:  This gate set then generates $\Rub_G(A^k)$
for each $k > 2$.)  Then a certificate accepted by $Z$ is a solution to
the constraint satisfaction problem
\[ x \in I^n \qquad \mbox{and} \qquad Z(x) \in F^n, \]
where $n$ is the width of $Z$.  The counting problem $\shZSAT_{G,A,I,F}$
counts the number of such solutions to $Z$.

We will need the following technical result from our previous work.

\begin{theorem}[{\cite[Lem.~4.1]{K:zombies}}] $\shZSAT_{G,A,I,F}$ is almost
parsimoniously $\shP$-complete.  If $p$ is a counting problem in $\shP$,
then there is a polynomial-time reduction $f \in \FP$ such that
\[\shZSAT_{G,A,I,F}(f(x)) = \#G \cdot p(x) + 1 \]
for every instance $x$ of $p$, where the $+1$ term accounts for the trivial,
all zombie solution $(z,\dots,z)$.  More precisely, the number of free
orbits of nontrivial solutions is parsimoniously $\shP$-complete.
\label{th:zsat} \end{theorem}

\begin{remark} In our previous work, we did not put the zombie symbol
$z$ in the sets $I$ and $F$, instead setting the initial and final sets
to be $(I \cup \{z\})^n$ and $(F \cup \{z\})^n$.  We also assumed the
side conditions that $I \ne F$, that $\#A \ge 2\#(I \cup F) + 3\#G + 1$,
and that $\#I, \#F \ge 2\#G$.  The first two of these conditions were
recognized as optional, but in fact they are all optional.  We can emulate
the first condition by adding a layer of unary gates at the beginning or
end, and we can attain the other two conditions by replacing $(A,I,F)$
by $(A^k,I^k,F^k)$ for some constant $k$.
\end{remark}

\section{Braid group actions}
\label{s:refine}

The main goal of this section is \Thm{th:refine}.  This theorem is a
refinement, in the special case that $G$ is simple, of a result of Roberts
and Venkatesh \cite[Thm.~5.1]{RV:hurwitz}.

\subsection{Conjugacy-restricted homomorphisms and colored braid subgroups}
\label{ss:actions}

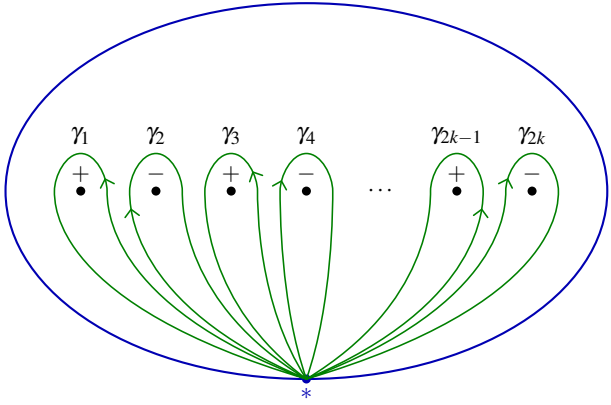
\begin{figure}
\begin{tikzpicture}[semithick,decoration={markings,
    mark=at position 0.42 with {\arrow{angle 90}}}]
\draw[thick,darkblue] (0,0) circle [x radius=4,y radius=2.5];
\coordinate (*) at (0,-2.5);
\fill[darkblue] (*) circle[radius=0.06];
\draw[darkblue] (*) node[below] {$*$};
\foreach \i in {-3,-2,-1,0,2,3} {
	\fill (\i,0) circle[radius=0.06]; }
\foreach \i in {-3,-1,2} {
	\draw (\i,0) node[above] {$+$};
	\draw[darkgreen,postaction={decorate}] (*) .. controls (\i+.35,-1.5) and
        (\i+.35, -0.2) .. (\i+.35,0) .. controls (\i+.35,.2) and (\i+.2,.5)
        .. (\i,.5) .. controls (\i-.2,.5) and (\i-.35,.2) .. (\i-.35,0)
        .. controls (\i-.35,-.2) and (\i-.35,-1.5) .. (*) ; }
\draw (-3,0.5) node[above] {$\gamma_{1}$};
\draw (-1,0.5) node[above] {$\gamma_{3}$};
\draw (2,0.5) node[above] {$\gamma_{2k-1}$};
\foreach \i in {-2,0,3} {
    \draw (\i,0) node[above] {$-$};
    \draw[darkgreen,postaction={decorate}] (*) .. controls (\i-.35,-1.5) and 
        (\i-.35, -0.2) .. (\i-.35,0) .. controls (\i-.35,.2) and (\i-.2,.5)
        .. (\i,.5) .. controls (\i+.2,.5) and (\i+.35,.2) .. (\i+.35,0)
        .. controls (\i+.35,-.2) and (\i+.35,-1.5) .. (*); }
\draw (-2,0.5) node[above] {$\gamma_{2}$};
\draw (0,0.5) node[above] {$\gamma_{4}$};
\draw (3,0.5) node[above] {$\gamma_{2k}$};	
\draw (1,0) node {$\cdots$};
\end{tikzpicture}
\caption{Generators for $\pi_1(D^2\setminus [2k])$.}
\label{f:generators} \end{figure}

Recall that $D^2 \setminus [2k]$ denotes the disk $D^2$ minus a set
$[2k]$ of $2k$ points.  As shown in \Fig{f:generators}, place the points in a
line and alternately label them $+$ and $-$, and choose a base point $*$
that is not on the same line.  Also as shown, choose a list of generators
of $\pi_1(D^2 \setminus [2k])$ represented by simple closed curves
$\gamma_1,\dots,\gamma_{2k}$, where each $\gamma_i$ winds counterclockwise
around the puncture $p_i$ when it is positive (when $i$ is odd) and clockwise
when it is negative (when $i$ is even).  Finally we choose one more curve
\[ \gamma_\infty = \gamma_{2k}^{-1} \gamma_{2k-1} \cdots
    \gamma_4^{-1} \gamma_3 \gamma_2^{-1} \gamma_1 \]
representing the boundary of the disk.  Note that we concatenate from left to
right, \eg, $\gamma_1 \gamma_2$ is the element of $\pi_1(D^2 \setminus [2k])$
that first traverses $\gamma_1$ and then $\gamma_2$.

When $G$ is a finite group and $C \subset G$ is a union of conjugacy
classes, we define three sets of homomorphisms from $\pi_1(D^2 \setminus
[2k])$ to $G$:
\begin{align*}
T_k(G,C) &\defeq \{ f:\pi_1(D^2 \setminus [2k])
    \to G \mid f(\gamma_i) \in C \} \\
\hR_k(G,C) &\defeq \{ f \in T_k(G,C) \mid f(\gamma_\infty) = 1 \} \\
R_k(G,C) &\defeq \{f \in \hR_k(G,C) \mid f \text{ is onto}\}.
\end{align*}
When $G$ and $C$ are clear from context, we will also use the abbreviations
\[ T_k \defeq T_k(G,C) \qquad \hR_k \defeq \hR_k(G,C)
    \qquad R_k \defeq R_k(G,C). \]
Note that $C$ will often but not always be a single conjugacy class.

Since $\pi_1(D^2 \setminus [2k])$ is freely generated by
$\gamma_1,\gamma_2,\dots,\gamma_{2k}$, the set $T_k(G,C)$ is bijective
with $C^{2k}$.  By abuse of notation, we will sometimes specify elements
of $T_k(G,C)$ as lists of elements in $C$.

Since $f(\gamma_\infty) = 1$ in both $\hR_k$ and $R_k$, each such $f$
factors through $\pi_1(S^2 \setminus [2k])$, where $S^2 \setminus [2k]$
is a punctured sphere obtained by collapsing the boundary $\partial D^2$
to a point.  Also by abuse of notation, we will also interpret each such $f$
as having this domain instead.

The homomorphism sets $T_k(G,C)$, $\hR_k(G,C)$, and $R_k(G,C)$ are all
invariant under the colored braid group $B_{k,k} \le B_{2k}$, by definition
the subgroup of the braid group that preserves the labels $+$ and $-$ of
the $2k$ punctures.  The goal of this section is to show that the action
of $B_{k,k}$ is large enough that we can implement gates in $\ZSAT$ with it.

\subsection{Invariant homology classes and the Conway-Parker theorem}
\label{ss:CP}

In this subsection, let $G$ be any finite group which is generated by a
single conjugacy class $C$.  We describe the orbits of the $B_{k,k}$ action
on $R_k$ in the limit as $k \to \infty$.  We will say that a property of
the action holds \emph{eventually} if it is a stable property in this limit,
\ie, if it is true for all $k$ large enough.

The main tool that we need is the Brand classifying space $B(G,C)$
\cite{Brand:branched}.  This space is a modification of the usual
classifying space $B(G)$ (often written $BG$) of a group $G$ ``relative" to
a conjugacy class $C$.  It has a number of important properties for our
purposes, some described by Brand, and some given by the second author
\cite{Samperton:schur}.  Before stating these properties, we first review
the definition.  The free loop space $LB(G)$ comes with an evaluation map
\[ \ev:LB(G) \times S^1 \to B(G), \]
and it has a connected component $L_CB(G)$ whose loops represent the chosen
conjugacy class $C \subseteq G$.  We define $B(G,C)$ by gluing $B(G)$ to
$L_CB(G) \times D^2$ using the evaluation map:
\begin{eq}{e:brand} B(G,C) \defeq (B(G) \sqcup L_CB(G) \times D^2)/\ev.
\end{eq}
In other words, an element $f \in L_CB(G)$ is also a continuous function
$f:S^1 \to B(G)$, and we identify $(f,x) \in L_CB(G) \times D^2$ with $f(x)
\in B(G)$ when $x \in S^1 = \partial D^2$.  We also retain a base point for
$B(G)$ and thus $B(G,C)$, even though we use the free loop space rather
than the based loop space to define the latter.

Recall also that the homology of a group is by definition the homology
of a classifying space, $H_*(G) = H_*(B(G))$.

\begin{remark} Following Ellenberg, Venkatesh, and Westerland
\cite{EVW:hurwitz2}, Roberts and Venkatesh \cite{RV:hurwitz} extend the
notation $H_*(G)$ in an ad hoc way to a reduced Schur multiplier that they
denote ``$H_2(G,C)$".  We will later denote the reduced Schur multiplier
as $M(G,C)$ instead.  It is a subgroup of $H_2(B(G,C))$, which we will
also not abbreviate as ``$H_2(G,C)$".   The reason is that we do not know
a natural interpretation of either $M(G,C)$ or $H_*(B(G,C))$ as a relative
homology group.
\end{remark}

\begin{proposition}[\cite{Samperton:schur}] Let $B(G,C)$ be the Brand
classifying space for a finite group $G$ generated by a conjugacy class
$C$.  For each $a$ and $b$, let $\Sigma_{a,b} = S^2 \setminus [a+b]$
be a punctured sphere with $a$ points marked $+$, $b$ points marked $-$,
and a base point $*$.  Then:
\begin{enumerate}
\item Let $f:\pi_1(\Sigma_{a,b}) \to G$ be a group homomorphism such that
each $+$ point has counterclockwise monodromy in $C$ and each $-$ point
has clockwise monodromy in $C$.  Then $f$ is represented by a pointed map
$\phi:S^2 \to B(G,C)$.
\item Every pointed map $\phi:S^2 \to B(G,C)$ in general position yields
a homomorphism $f:\pi_1(\Sigma_{a,b}) \to G$ as in part 1, for some $a$
and $b$.   The maps $\phi_0 \sim \phi_1$ are homotopic if and only if the
homomorphisms $f_0$ and $f_1$ are connected by a concordance
\[ f:\pi_1((S^2 \times I) \setminus L) \to G, \]
where $L$ is a tangle.
\item Given that $C$ generates $G$, $B(G,C)$ is simply connected.
\item There is an exact sequence
\[ Z(c)_\ab \stackrel{\kappa}\longto H_2(G) \stackrel{\beta}{\longto}
    H_2(B(G,C)) \stackrel{\sigma}{\longto} \Z \to 0, \]
where $Z(c) \subseteq G$ is the centralizer of any one element $c \in C$.
Given that $G$ is finite, the image of $\beta$ is $H_2(B(G,C))_\tor$.
\item Given $f$, $\phi$, and $\sigma$ from the previous,
\[ (\sigma \circ \phi_*)([S^2]) = a-b \in \Z. \]
\end{enumerate}
\label{p:brand} \end{proposition}

\begin{remark} The Brand classifying space $B(G,C)$ exists for any $G$
(not necessarily finite) and any union of conjugacy classes (not
just one, and not necessarily generating $G$).  In general, the homotopy
classes $[M,B(G,C)]$ from a smooth manifold $M$ of any dimension classify
the concordance classes of $C$-branched $G$-covers of $M$, such that the
codimension 2 branch locus has a distinguished normal framing.  Likewise the
cobordism groups $\Omega_n(B(G,C))$ classify the cobordism classes of such
branched coverings of $n$-manifolds.
\end{remark}

Following \Prop{p:brand}, we introduce the abuse of notation
\[ f_* = \phi_*:H_2(S^2) \to H_2(B(G,C)).\]
We also define
\[ M(G,C) \defeq H_2(B(G,C))_\tor. \]
The group $M(G,C)$, the \emph{reduced Schur multiplier}, is a quotient of
the usual Schur multiplier $M(G) = H_2(G)$.  Now suppose that
\[ f:\pi_1(S^2 \setminus [2k]) \to G \]
is an element of $\hR_k$, with $k = a = b$.  Then
\[ f_*([S^2]) \in H_2(B(G,C)) \]
maps to zero in $H_2(B(G,C))_\free$.  It thus lies in $M(G,C)$.  We define the
(branched) \emph{Schur invariant} of $f$ to be
\[ \sch(f) \defeq f_*([S^2]) \in M(G,C). \]
By construction, the function
\[ \sch:\hR_k \to M(G,C) \]
is invariant under the action of the colored braid group $B_{k,k}$.

\begin{theorem}[Ellenberg-Venkatesh-Westerland
    {\cite[Thm. 1.1]{Samperton:schur}}]
Let $G$ be a finite group generated by a conjugacy class $C$.
Then, eventually, the Schur invariant yields a bijection
\[\sch:R_k/B_{k,k} \stackrel{\cong}{\longto} M(G,C). \]
In particular, $\sch$ is eventually injective; \ie, it is eventually a
complete orbit invariant for the action of $B_{k,k}$ on $R_k$.
\label{th:evw} \end{theorem}

\begin{remark} \Thm{th:evw} first appeared in version 1 of an arXiv e-print
by Ellenberg, Venkatesh, and Westerland \cite{EVW:hurwitz2}.  This e-print
was later withdrawn for unrelated reasons, but (besides that arXiv versions
are permanent) the argument was later cited and sketched by Roberts and
Venkatesh \cite{RV:hurwitz}.   The second author \cite{Samperton:schur}
then found a topological proof of the same result using the Brand
classifying space.  The new results of \cite{Samperton:schur} also hold for
surfaces with either genus or punctures or both, and thus subsume a result of
Dunfield and Thurston \cite[Thm.~6.23]{DT:random}.   In fact, \Thm{th:evw} also
holds when $C$ is a union of conjugacy classes rather than just one.  (In full
generality, the Schur invariant $\sch(f)$ lies in a torsor of the reduced
Schur multiplier $M(G,C)$ rather than directly in this abelian group.)
The original result along these lines is the unpublished Conway-Parker
theorem, which is the case $C = G$, and which was later proven in the
literature by Fried and V\"olklein \cite{FV:inverse}.
\end{remark}

We will use two basic properties of the Schur invariant, one of which
requires a definition:  Say that $f \in \hR_k$ \emph{bounds a plat} if
there is an inclusion $S^2 \setminus [2k] \to B^3$ such that $f$ extends
to a homomorphism from the fundamental group of the complement of a trivial
tangle in $B^3$ with oriented arcs.

\begin{lemma}
The Schur invariant has the following properties.
\begin{enumerate}
\item If $f \in \hR_a$ and $g \in \hR_b$, and $f\;\#\;g \in \hR_{a+b}$
is their boundary sum, then $\sch(f\;\#\;g) = \sch(f) + \sch(g)$.
\item If $f \in \hR_k$ bounds a plat, then $\sch(f) = 0$.
\end{enumerate}
\label{l:schprops} \end{lemma}

\begin{proof}
Part 1: The boundary sum of $f$ and $g$ corresponds to the group law in
\[ \pi_2(B(G,C)) = [S^2:B(G,C)]. \]
Thus:
\begin{align*} \sch(f\;\#\;g) &= (f\;\#\;g)_*([S^2])
    = f_*([S^2]) + g_*([S^2]) \\ &= \sch(f) + \sch(g).
\end{align*}

Part 2: The map $f$ is null-concordant by hypothesis, hence $\sch(f) = 0$
is null-homologous by \Prop{p:brand}.
\end{proof}

Our reduction from $\ZSAT$ to $\#H(G,C)$ makes special use
of maps $f \in \hR_k$ with $\sch(f) = 0$.  Hence we define
\begin{align*}
\hR_k^0 &\defeq \{ f \in \hR_k \mid \sch(f) = 0 \} \\
R_k^0 &\defeq \{ f \in R_k \mid \sch(f) = 0 \}.
\end{align*}

\subsection{The perfect case}
\label{ss:perfect}

In this subsection, we establish some further properties of the reduced Schur
multiplier $M(G,C)$ with the additional assumption that $G$ is perfect.
Not all of the properties require this hypothesis, but everything listed
is at least better motivated in that case.  We begin with the following
interpretation of $M(G,C)$ which is explained by Roberts and Venkatesh
\cite[Sec.~4B]{RV:hurwitz}. If $G$ is a perfect group, then it has a
canonical central extension
\[ M(G) \into \hG \onto G \]
called the \emph{Schur cover} $\hG$ of $G$.  In general the conjugacy
classes in $\hG$ can be larger than their counterparts in $G$, in the
sense that two preimages $g_1,g_2 \in \hG$ of one element $g \in G$ can be
conjugate to each other.   The reduced multiplier $M(G,C)$ is the finest
possible quotient of $M(G)$ such that in the corresponding central extension
\[ M(G,C) \into \tG \onto G, \]
two distinct preimages $c_1, c_2 \in \tG$ of $c \in C$ are never conjugate.
In other words, if $C' \subseteq \tC$ is any one conjugacy class
in the preimage $\tC$ of $C$, then $\tC$ decomposes as
\[ \tC = M(G,C) \cdot C', \]
where each $mC'$ with $m \in M(G,C)$ is a distinct conjugacy class.

\begin{lemma} If $C$ generates $G$ and $G$ is perfect, then
\[ \lim_{k \to \infty} \frac{\#R_k}{(\#C)^{2k}} = \frac{1}{\#G} \qquad
    \lim_{k \to \infty} \frac{\#R^0_k}{(\#C)^{2k}}
    = \frac{1}{\#G \cdot \#M(G,C)}. \]
\label{l:stochastic} \end{lemma}

We note Dunfield and Thurston proved a version of this lemma for maps
from fundamental groups of closed surfaces, instead of punctured disks.
Our limits are analogs of \cite[Lems.~6.10~\&~6.13]{DT:random}.

\begin{proof} For the first claim, we consider an infinite list
\[ c_1,c_2,c_3,\ldots \in C \]
of elements of $C$ chosen independently and uniformly at random.
The first $2k$ of these elements describe a homomorphism
\[ f_k:\pi_1(D^2 \setminus [2k]) \to G \]
with $f_k \in T_k$ following the conventions in \Sec{ss:actions}.
Then $f_k \in R_k$ when the product
\[ g_k \defeq c_{2k}^{-1} c_{2k-1} \ldots c_3 c_2^{-1} c_1 \]
equals 1, since $g_k = f_k(\gamma_\infty)$.  The first limit can thus be
restated as saying that the probability that $g_k = 1$ converges to $1/\#G$
as $k \to \infty$.  To argue this, we note the inductive relation
\[ g_k = c_{2k}^{-1}c_{2k-1}g_{k-1}, \]
and we let $M$ be the corresponding stochastic transition matrix,
independent of $k$, on probability distributions on $g_k$ drawn from
$G$. One can check these three properties of $M$:
\begin{enumerate}
\item[1.] $M$ commutes with both left and right multiplication by $G$.
\item[2.] $M$ is symmetric, $M = M^T$, and thus doubly stochastic.
\item[3.] Each diagonal entry of $M$ equals $1/\#C$.
\end{enumerate}

We apply the Perron-Frobenius theorem to the matrix $M$, in the doubly
stochastic case.  By this theorem, either $M^k$ converges to a constant
matrix as $k \to \infty$, or $G$ has a non-trivial equivalence relation
$\sim$ such that $M$ descends to a permutation on the quotient set $G/\sim$.
Since the diagonal of $M$ is entirely non-zero, this permutation must be
the identity.  Since $M$ commutes with both left and right multiplication
by every $g \in G$, the set quotient $G/\sim$ must be a group quotient
$G/N$ by some normal group $N \normaleq G$.  The conjugacy class $C$
descends to a conjugacy class $E$ which generates $G/N$.  Since $M$ acts
by the identity on $G/N$, $E$ must have a single element.  Thus $G/N$ would
be a cyclic group if it existed, contradicting that $G$ is perfect.

This establishes the first limit, except with a numerator of $\#\hR_k$
rather than $\#R_k$.  For the rest of the limit, observe that in the given
random process, the image of $f_k$ is monotonic, more precisely that it
almost surely increases to $G$ and stays there.  By contrast the condition
$g_k = 1$ is recurrent.  Therefore the limiting probability that $f_k \in
\hR_k$ is the same as the limiting probability that $f_k \in R_k$.

We reduce the second limit to the first one.  Recall that $\tG$ is the
central extension of $G$ by $M(G,C)$, and that $\tG$ is a perfect group
because the full Schur cover $\hG$ of the perfect group $G$ is perfect.
Let $C' \subseteq \tG$ be a conjugacy class that lifts $C$.  (Any such
lift generates $\tG$.)  Then $f \in R_k$ has a lift $f' \in R_k(\tG,C')$,
and we can recognize the Schur invariant $\sch(f)$ as
\[ \sch(f) = f'(\gamma_\infty), \]
independent of the choice of $C'$.  Therefore the second limit for the
group $G$ is equivalent to the first limit for the group $\tG$, as desired.
\end{proof}

The remaining properties concern the Cartesian power $(G^\ell,C^\ell)$
of $(G,C)$ and require some algebraic topology to state properly.
If we identify $B(G^\ell)$ with $B(G)^\ell$, then this identification
extends to a natural map
\[ \psi:B(G^\ell,C^\ell) \to B(G,C)^\ell\]
in the following way.  By the definition of $B(G,C)$, equation \eqref{e:brand},
we need to describe a map
\[ \psi_1:L_{C^\ell}B(G^\ell) \times D^2 \to (L_CB(G) \times D^2)^\ell \]
that commutes with the evaluation maps.  For this purpose, we let
$\psi$ be the product of two maps
\[ \psi_2:L_{C^\ell}B(G^\ell) \stackrel{\cong}\longto
L_CB(G)^\ell \qquad \Delta:D^2 \to (D^2)^\ell. \]
The map $\psi_2$ is another natural isomorphism, while $\Delta$ is the
diagonal embedding.  With these choices, $\psi_1 = \psi_2 \times \Delta$
commutes with the evaluation map, completing the construction of $\psi$.

\begin{lemma} Let $\ell > 0$ be an integer and assume that $C$
generates $G$ and $G$ is perfect.  Then:
\begin{enumerate}
\item $C^\ell$ generates $G^\ell$.
\item The K\"unneth theorem yields the isomorphism
\[ H_2(B(G,C)^\ell) \cong H_2(B(G,C))^\ell. \]
\item The map $\psi$ commutes with the natural equivalence
\[ \hR_k(G^\ell,C^\ell) \cong \hR_k(G,C)^\ell. \]
Using part 1, this equivalence also commutes with the Schur invariant:
\[ \sch((f_1,f_2,\ldots,f_\ell))
    = (\sch(f_1),\sch(f_2),\ldots,\sch(f_\ell)). \]
\item The induced map
\[ \psi_*:H_2(B(G^\ell,C^\ell)) \to H_2(B(G,C)^\ell)
    \cong H_2(B(G,C))^\ell \]
is injective.
\end{enumerate}
\label{l:kunneth} \end{lemma}

\begin{proof} Part 1: $C^\ell$ generates a normal subgroup $N \leq G^\ell$
that surjects onto each factor, so \Lem{l:ribetn} tells us that $N = G^\ell$.

Part 2: \Prop{p:brand}, part 3, says that $B(G,C)$ is simply connected
when $C$ generates $G$, in particular that
\[ H_1(B(G,C)) = 0. \]
Moreover, $H_0(B(G,C)) = \Z$ since $B(G,C)$ is connected.  Thus the
K\"unneth theorem simplifies to the stated isomorphism.

Part 3: The main step is to review the construction of a map
\[ \phi:S^2 \to B(G,C) \]
representing $f \in \hR_k(G)$, and to then relate $\phi$ to
the map $\psi$, as promised in \Prop{p:brand}.  Given
\[ f:\pi_1(S^2 \setminus [2k]) \to G, \]
we remove $2k$ open disks from $S^2$ instead of just $2k$ point punctures
to obtain a surface $S^2 \setminus kD^2$ with $k$ boundary circles
around the punctures.  We can define a map
\[ \phi:S^2 \setminus [2k] \to B(G) \]
using the map $f$, and then use a fiber $D^2$ from the attachment $L_CB(G)
\times D^2$ to extend $\phi$ across each puncture.

If we replace $(G,C)$ by $(G^\ell,C^\ell)$ in this construction, then it
commutes with $\phi$ because the same extension disk in $S^2$ is used $\ell$
times for each puncture.  Moreover, if $\psi_i$ is the $i$th component
of the map $\psi$, then the composition $\phi_i = \psi_i \circ \phi$
matches the $i$th component $f_i$ of $f$.  Together with the proof of
the K\"unneth formula and its use in part 1, this establishes that $i$th
component of $\sch(f)$ is $\sch(f_i)$, as desired.

Part 4: Let $c \in C$ and consider the diagram
\[ \begin{tikzcd}[column sep=scriptsize]
Z(c^{\times \ell})_\ab \arrow{r}{\kappa}\arrow{d}{\cong} & H_2(G^\ell)
    \arrow{r}{\beta}\arrow{d}{\cong} & H_2(B(G^\ell,C^\ell))
    \arrow{r}{\sigma}\arrow{d}{\psi_*} & \Z \arrow{d}{\Delta} \\
    Z(c)_\ab^\ell \arrow{r}{\kappa^{\times \ell}} & H_2(G)^\ell
    \arrow{r}{\beta^{\times \ell}} & H_2(B(G,C))^\ell
    \arrow{r}{\sigma^{\times \ell}} & \Z^\ell, \end{tikzcd} \]
where each row is taken from part 4 of \Prop{p:brand} and is thus exact.
Meanwhile the first vertical map is the elementary isomorphism from group
theory; the second map is from the K\"unneth theorem and is an isomorphism
because $G$ is perfect; the third map is as indicated; and the fourth is
the diagonal embedding of $\Z$ into $\Z^\ell$.

We claim that the diagram is commutative.  Working from the left, the first
square commutes by the definition of the map $\kappa$ \cite{Samperton:schur}.
Given any group homomorphism $f:A \times B \to G$, there is always
a bi-additive map
\[ f_{**}:A_\ab \times B_\ab \cong H_1(A)
    \times H_1(B) \to H_2(A \times B) \stackrel{f_*}{\longto} H_2(G), \]
where the middle arrow is the K\"unneth map.  The map $\kappa$ is a linear
restriction of $f_{**}$ where $A = Z(c)$ and $B = \langle c\rangle$.  It is
easy to confirm that $\kappa$ for the group $G^\ell$ does the same thing as
$\kappa^{\times \ell}$ for the group $G$.  The second square is commutative
by the way that $\psi$ is constructed: Since it is the identity on $B(G^\ell)
= B(G)^\ell$, $\beta$ and $\beta^{\times \ell}$ also do the same thing.
Finally the third square commutes because $\sigma$ and $\sigma^{\times \ell}$
do the same thing by part 5 of \Prop{p:brand}.  By the Hurewicz theorem,
we can represent any element of $H_2(B(G^\ell,C^\ell))$ by a map from $S^2$
and thus by a homomorphism
\[ f:\pi_1(\Sigma_{a,b}) \to G^\ell. \]
Splitting this homomorphism $f$ into $\ell$ homomorphisms to $G$, the
difference $a-b$ is replicated $\ell$ times.

To complete the proof, since the diagram is commutative, the four lemma
says that $\psi_*$ is injective.
\end{proof}

\subsection{An ultra transitivity theorem}

\Thm{th:evw} says that the action of $B_{k,k}$ on $R_k^0$ is transitive for
all $k$ large enough.  Our goal now is \Thm{th:refine}, which, among other
things, gives a complete description of this action when $G$ is nonabelian
simple.  The structure of our argument is similar to one direction of the
full monodromy theorem of Roberts-Venkatesh \cite[Thm.~5.1]{RV:hurwitz}.
However, \Thm{th:refine} refines this special case of Roberts-Venkatesh
in the same way that our prior result \cite[Thm.~5.1]{K:zombies} refines
a result of Dunfield-Thurston \cite[Thm.~7.4]{DT:random}.

From here to the end of this article, we choose an element $c \in C$
and we declare the abbreviations
\[ U \defeq \Aut(G,c) \subseteq \hU \defeq \Aut(G,C). \]
In this subsection we will only use $\hU$, but we will need the subgroup
$U$ soon enough.  The group $\hU$ acts on $\hR_k^0$ because we can compose
a homomorphism
\[ f:\pi_1(D^2 \setminus [2k]) \to G \]
with an element $\alpha \in \Aut(G,C)$.   It acts freely on the subset
$R_k^0$ because these homomorphisms are surjective.  Moreover, the actions
of $B_{k,k}$ and $\Aut(G,C)$ on $R_k^0$ commute.  In other words, the
corresponding permutation representation is a map
\[ \rho: B_{k,k} \to \Sym_{\hU}(R_k^0). \]

\begin{lemma} Let $G$ be a nonabelian simple group, let $C\subseteq G$
be a conjugacy class, and let $\ell > 0$.  Then $B_{k,k}$ eventually
(as $k \to \infty$) acts $\hU$-set $\ell$-transitively on $R_k^0$.
\label{l:trans} \end{lemma}

\begin{proof} We choose $k$ large enough so that the conclusion
of \Thm{th:evw} holds for the finite group $G^\ell$ and the conjugacy class
$C^\ell$.  Let
\[ f_1,f_2,\dots,f_\ell \in R_k^0 \]
lie in distinct $\hU$-orbits and consider the product homomorphism
\[ f = f_1 \times f_2 \times \dots \times f_\ell: \pi_1((D^2 \setminus
     [2k])^\ell) \to G^\ell. \]
By \Lem{l:goursrib}, $f$ is surjective.  Since $\sch(f_j) = 0$ for all
$j=1,\dots,\ell$, \Lem{l:kunneth} implies that $\sch(f)=0$.  If
\[ e_1,e_2,\dots,e_\ell \in R_k^0 \]
is another such list of homomorphisms with the same properties with product
$e$, then \Thm{th:evw} says $e$ and $f$ are in the same orbit of $B_{k,k}$,
as desired.
\end{proof}

Besides the map $\rho$ already defined, let
\[ \sigma_{J,E}:B_{k,k} \to \Sym(T(J,E)) \]
be the action map of the braid group for every finite group $J$
generated by a set of conjugacy classes $E \subseteq J$.  Also let
\[ \phi:B_{k,k} \to \Sym(k)^2 \]
be the forgetful map that only remembers the permutation of the braid
strands.

\begin{theorem} Let $G$ be a finite, nonabelian, simple group, and let $C
\subseteq G$ be a conjugacy class. Then the image of $B_{k,k}$ under the
joint homomorphism $\rho \times \sigma \times \phi$:
\begin{multline*}
\rho \times \phi \times \!\!\! \prod_{\substack{J \supseteq E \\ \#E < \#C}}
     \!\!\! \sigma_{J,E}:B_{k,k} \longto \\ \Sym_\hU(R_k^0) \times
     \Sym(k)^2 \times \!\!\! \prod_{\substack{J \supseteq E \\ \#E < \#C}}
     \!\!\! \Sym(T(J,E))
\end{multline*}
eventually contains $\Rub_\hU(R_k^0)$, where here each group $J$ is
generated by a union of conjugacy classes $E \subseteq J$.
\label{th:refine} \end{theorem}

In other words, we can find a set of \emph{pure} braids that act by the
entire Rubik group on $R_k^0$, while simultaneously acting trivially on
every $T(J,E)$ with $\#E < \#C$.   It follows that such braids also act
trivially on the set $\hR_k \setminus R_k$ of non-surjective maps in $\hR_k$.

\begin{proof} Following Dunfield-Thurston \cite{DT:random} and
Roberts-Venkatesh \cite{RV:hurwitz}, we use the corollary of the
classification of finite simple groups that every 6-transitive permutation
group on a finite set is ultratransitive.  \Lem{l:trans} shows $B_{k,k}$
eventually acts $\Aut(G,C)$-set 6-transitively on $R_k^0$.  It follows
that $B_{k,k}$ eventually acts 6-transitively (in the usual sense)
on $R_k^0/\Aut(G,C)$.  By \Thm{th:rubik}, the image of $\rho$ contains
$\Rub_\hU(R_k^0)$.

For the rest of the properties of the joint homomorphism, \Lem{l:zorn}
tells us it is enough to show that $\Rub_\hU(R_k^0)$ does not have any simple
quotients that are subquotients of
\[ \Sym(k)^2 \times \!\!\!  \prod_{\substack{J \supseteq E \\ \#E < \#C}}
     \!\!\! \Sym(T(J,E)). \]
By Lemmas \ref{l:only} and \ref{l:prodquo}, it suffices to show that
$\Alt(R_k^0/\hU)$ is not a subquotient of $\Sym(T(J,E))$ for any group
$J$ generated by a union of conjugacy classes $E$, nor a subquotient of
$\Sym(k)$.  To prove this, we show that $\Alt(R_k^0/\hU)$ is eventually
larger than any of these other groups.  This follows from comparing
$\#T(J,E) = \#E^{2k}$ to the bound in \Lem{l:stochastic} that shows that
\[ \lim_{k \to \infty} (\#R_k^0)^{1/2k} = \#C. \]
Meanwhile $\Sym(k)$ by definition acts on a set that only grows linearly
in $k$, not exponentially.
\end{proof}

\begin{remark} Although our proof of \Thm{th:refine} (hence also our main
theorem) depends on the classification of finite simple groups via the
6-transitivity corollary, we conjecture that the classification can
be avoided.  The analogous step in our previous work \cite{K:zombies} is
a result of Dunfield and Thurston \cite[Thm.~7.4]{DT:random} that they
argue in the same way.   However, they point out that they could use a
non-classification result of Dixon and Mortimer \cite[Thm.~5.5B]{DM:perm},
which says that a permutation group on a finite set is ultratransitive
if it is both 2-transitive, and locally $\ell$-transitive on a
single subset of size $\ell = \Omega(\log k)$.  They show that this
transitivity theorem suffices (for mapping class groups of unmarked,
closed surfaces) when the Schur invariant vanishes thanks to a result of
Gilman \cite{Gilman:quotients}, but the argument can be extended to any
value of the Schur invariant.  It would suffice to find an analogue of
Gilman's theorem for braid groups.
\end{remark}

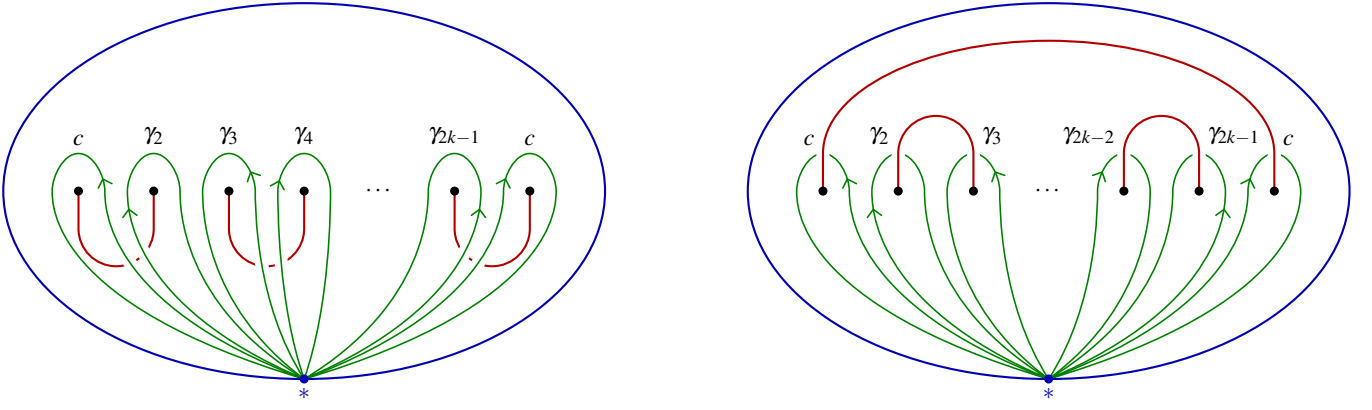
\begin{figure*}
\begin{center}
\begin{tikzpicture}[semithick,decoration={markings,
    mark=at position 0.42 with {\arrow{angle 90}}}]
\coordinate (*) at (0,-2.5);
\draw[darkred,thick] (-3,0) -- (-3,-.5) arc(180:360:0.5) -- (-2,0);
\draw[darkred,thick] (-1,0) -- (-1,-.5) arc(180:360:0.5) -- (0,0);
\draw[darkred,thick] (2,0) -- (2,-.5) arc(180:360:0.5) -- (3,0);
\foreach \i in {-3,-2,-1,0,2,3} {
	\fill (\i,0) circle[radius=0.06]; }
\foreach \i in {-3,-1,2} {
	\draw[white,line width=5] (*) .. controls (\i+.35,-1.5) and
        (\i+.35, -0.2) .. (\i+.35,0) .. controls (\i+.35,.2) and (\i+.2,.5)
        .. (\i,.5) .. controls (\i-.2,.5) and (\i-.35,.2) .. (\i-.35,0)
        .. controls (\i-.35,-.2) and (\i-.35,-1.5) .. (*); }
\foreach \i in {-2,0,3} {
	\draw[white,line width=5] (*) .. controls (\i-.35,-1.5) and
        (\i-.35, -0.2) .. (\i-.35,0) .. controls (\i-.35,.2) and (\i-.2,.5)
        .. (\i,.5) .. controls (\i+.2,.5) and (\i+.35,.2) .. (\i+.35,0)
        .. controls (\i+.35,-.2) and (\i+.35,-1.5) .. (*) ; }
\foreach \i in {-3,-1,2} {
	\draw[darkgreen,postaction={decorate}] (*) .. controls (\i+.35,-1.5) and
        (\i+.35, -0.2) .. (\i+.35,0) .. controls (\i+.35,.2) and (\i+.2,.5)
        .. (\i,.5) .. controls (\i-.2,.5) and (\i-.35,.2) .. (\i-.35,0)
        .. controls (\i-.35,-.2) and (\i-.35,-1.5) .. (*); }
\draw (-3,0.5) node[above] {$c$};
\draw (-1,0.5) node[above] {$\gamma_{3}$};
\draw (2,0.5) node[above] {$\gamma_{2k-1}$};
\foreach \i in {-2,0,3} {
	\draw[darkgreen,postaction={decorate}] (*) .. controls (\i-.35,-1.5) and
        (\i-.35, -0.2) .. (\i-.35,0) .. controls (\i-.35,.2) and (\i-.2,.5)
        .. (\i,.5) .. controls (\i+.2,.5) and (\i+.35,.2) .. (\i+.35,0)
        .. controls (\i+.35,-.2) and (\i+.35,-1.5) .. (*); }
\draw (-2,0.5) node[above] {$\gamma_{2}$};
\draw (0,0.5) node[above] {$\gamma_{4}$};
\draw (3,0.5) node[above] {$c$};	
\draw (1,0) node {$\cdots$};
\fill[darkblue] (*) circle[radius=0.06];
\draw[darkblue] (*) node[below] {$*$};
\draw[thick,darkblue] (0,0) circle [x radius=4,y radius=2.5];
\end{tikzpicture}
\hfill
\begin{tikzpicture}[semithick,decoration={markings,
    mark=at position 0.42 with {\arrow{angle 90}}}]
\coordinate (*) at (0,-2.5);
\foreach \i in {-3,-1,2} {
	\draw[darkgreen,postaction={decorate}] (*) .. controls (\i+.35,-1.5)
        and (\i+.35, -0.2) .. (\i+.35,0) .. controls (\i+.35,.2) and
        (\i+.2,.5) .. (\i,.5) .. controls (\i-.2,.5) and (\i-.35,.2)
        .. (\i-.35,0) .. controls (\i-.35,-.2) and (\i-.35,-1.5) .. (*) ; }
\draw (-3,0.5) node[above left] {$c$};
\draw (-1,0.5) node[above right] {$\gamma_{3}$};
\draw (2,0.5) node[above right] {$\gamma_{2k-1}$};
\foreach \i in {-2,1,3} {
	\draw[darkgreen,postaction={decorate}] (*) .. controls (\i-.35,-1.5)
        and (\i-.35, -0.2) .. (\i-.35,0) .. controls (\i-.35,.2) and
        (\i-.2,.5) .. (\i,.5) .. controls (\i+.2,.5) and (\i+.35,.2)
        .. (\i+.35,0) .. controls (\i+.35,-.2) and (\i+.35,-1.5) .. (*) ; }
\draw (-2,0.5) node[above left] {$\gamma_{2}$};
\draw (1,0.5) node[above left] {$\gamma_{2k-2}$};
\draw (3,0.5) node[above right] {$c$};	
\draw (0,0) node {$\cdots$};
\fill[darkblue] (*) circle[radius=0.06];
\draw[darkblue] (*) node[below] {$*$};
\draw[thick,darkblue] (0,0) circle [x radius=4,y radius=2.5];
\draw[thick,white,line width = 5] (-3,0) -- (-3,.5) arc
    (180:0:3 and 1.5) -- (3,0);
\draw[thick,white,line width = 5] (-2,0) -- (-2,.5) arc (180:0:0.5) -- (-1,0);
\draw[thick,white,line width = 5] (1,0) -- (1,.5) arc (180:0:0.5) -- (2,0);
\draw[darkred,thick] (-3,0) -- (-3,.5) arc (180:0:3 and 1.5) -- (3,0);
\draw[darkred,thick] (-2,0) -- (-2,.5) arc (180:0:0.5) -- (-1,0);
\draw[darkred,thick] (1,0) -- (1,.5) arc (180:0:0.5) -- (2,0);
\foreach \i in {-3,-2,-1,1,2,3} {
	\fill (\i,0) circle[radius=0.06]; }
\end{tikzpicture} \hfill \end{center}
\caption{These two sets of red plats impose the initial and final
    constraints, respectively.}
\label{f:plats} \end{figure*}

\section{Proof of \Thm{th:main} }
\label{s:proof}

As before, let $G$ be a non-abelian simple group, let $C \subseteq G$
be a non-trivial conjugacy class (which necessarily generates $G$), and
let $c \in C$ be a distinguished element.  Again, let
\[ U \defeq \Aut(G,c) \subseteq \Aut(G,C) \defeq \hU. \]
Also for this section, choose some fixed $k$ large enough for
the conclusions of \Thm{th:refine}.

\begin{remark} Although our proof of \Thm{th:main} is similar to that of
our previous result for mapping class groups \cite{K:zombies}, we will
face a new technical difficulty.  Namely, even though the available braid
actions are $\hU$-equivariant, the reduction from $\ZSAT$ is locally only
$U$-equivariant.  We will define a zombie symbol $z$ using the distinguished
element $c$, which is not a $\hU$-invariant choice.  If we represented
$z$ using all of $C$ or in any other $\hU$-invariant manner, then the
construction could only produce an intractable link invariant rather than
specifically a knot invariant.
\end{remark}

\subsection{Alphabets and gadgets}
\label{ss:alphagadget}

In this subsection, we will define a $U$-set alphabet $A$ with subsets $I,F
\subseteq A$ and a zombie symbol $z \in I \cap F$ such that $\ZSAT_{U,A,I,F}$
satisfies \Thm{th:zsat} and is thus almost parsimoniously $\shP$-complete,
for use to the end of this article.  Then we will define pure braid gadgets
that we will later use to replace gates in a $\ZSAT$ circuit. Here a
\emph{gadget} is a semi-rigorous concept in theoretical computer science,
by definition a local combinatorial replacement to implement a complexity
reduction.  Our gadget to replace one gate will be a braid with a fixed
number of strands.  We will later concatenate these braid gadgets to replace
an entire circuit with a braid with a linear number of strands.

Let the zombie symbol be
\[ z \defeq (c,c,\dots,c) \in \hR_k^0, \]
and let the alphabet be
\[ A \defeq \{z\} \cup \{ (g_1,g_2,\dots,g_{2k}) \in R^0_k
    \mid g_1 = g_{2k} = c \}. \]
\Ie, the non-zombie symbols in $A$ are surjections with trivial Schur
invariant such that the first and last punctures map to $c$ specifically.
The initialization and finalization conditions are specified by
restricting to homomorphisms that factor through the two trivial tangles
in \Fig{f:plats}, respectively.  Precisely, we define the initial and
final subalphabets by
\begin{align*}
I &\defeq \{ (g_1,g_2,\dots,g_{2k}) \in A \mid
    g_{2i} = g_{2i-1} \; \forall i \le k \} \\
F &\defeq \{ (g_1,g_2,\dots,g_{2k}) \in A \mid
    g_{2i} = g_{2i+1}\; \forall i \le k-1 \}
\end{align*}
It is straightforward to verify that $U$, $A$, $I$, and $F$ satisfy the
conditions of \Thm{th:zsat}.

We now construct braid gadgets that simulate gates in $\Rub_U(A^2)$.
Let $D^2 \setminus [4k]$ be a pointed disk with $4k$ punctures, and divide
it into two half-disks by a straight line, so that each half contains the
base point and half of the $4k$ punctures, as in \Fig{f:sum}.  This allows
us to identify $T_k \times T_k \cong C^{2k} \times C^{2k}$  with $T_{2k}
\cong C^{4k}$.  It is straightforward to verify that this identification
takes $R_k^0 \times R_k^0$ to a subset of $R_{2k}^0$.  In particular,
we identify $A^2$ with a subset of $R_{2k}^0 \cup \{(z,z)\}$.

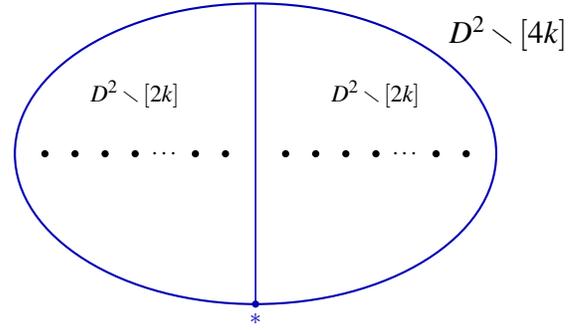
\begin{figure}
\begin{tikzpicture}[semithick,decoration={markings,
    mark=at position 0.42 with {\arrow{angle 90}}},scale=0.8]
\draw[thick,darkblue] (0,0) circle [x radius=4, y radius=2.5];
\coordinate (*) at (0,-2.5);
\fill[darkblue] (*) circle[radius=0.06];
\draw[darkblue] (*) node[below] {$*$};
\draw[darkblue] (*) -- (0,2.5);
\foreach \i in {-3.5,-3,-2.5,-2,-1,-.5} {
	\fill (\i,0) circle[radius=0.06]; }
\draw (-1.5,0) node {$\cdots$};
\foreach \i in {3.5,3,2,1.5,1,.5} {
	\fill (\i,0) circle[radius=0.06]; }
\draw (2.5,0) node {$\cdots$};
\draw (-2,1) node {$D^2 \setminus [2k]$};
\draw (2,1) node {$D^2 \setminus [2k]$};
\draw (4.2,2) node {\large $D^2 \setminus[4k]$};
\end{tikzpicture}
\caption{Splitting $D^2\setminus[4k]$ as a boundary sum of two copies
    of $D^2\setminus[2k]$.}
\label{f:sum}\end{figure}

\begin{corollary} For every gate $\delta \in \Rub_U(A^2)$, there is a
braid word $b(\delta)$, interpreted also as a braid element $b(\delta)
\in B_{2k,2k}$, with the following properties:
\begin{enumerate}
\item $b(\delta)$ acts on $A^2$ as $\delta$.
\item $b(\delta)$ acts trivially on $R_{2k}^0 \setminus \hU \cdot (A^2)$.
\item $b(\delta)$ acts trivially on $T_{2k}(J,E) \cong E^{4k}$ for every
group $J$ generated by a union of conjugacy classes $E$ with $\#E < \#C$.
\item $b(\delta)$ is a pure braid, \ie, $b(\delta) \in PB_{4k} \le B_{2k,2k}$.
\end{enumerate}
\label{c:gadget} \end{corollary}

The existence of $b(\delta)$ follows immediately from \Thm{th:refine}.
In fact, for each $\delta$ in $\Rub_U(A^2)$ there are infinitely many
braids that satisfy properties 1-4, but it is important for our reduction
that we fix some suitable $b(\delta)$ for each $\delta$.

Note that property 1 specifies the action of $b(\delta)$ on $A^2$,
but it implies more than that, because the action of $B_{2k,2k}$ is
$\hU$-equivariant while $A^2$ is only closed under the action of $U$.
This action has a unique $\hU$-equivariant extension to $\hU \cdot (A^2)$.
Meanwhile property 3 implies that $b(\delta)$ acts trivially on $\hR_{2k}
\setminus R_{2k}$, so together the first three properties specify all of
the action of $b(\delta)$ on $\hR_{2k}^0$.   However, $b(\delta)$ is not
fully specified on all of $\hR_{2k}$, because we place no restrictions on
its effect on $f \in R^s_{2k}$ with non-vanishing Schur invariant, $s \ne 0$.

We record as a lemma several invariance properties of $b(\delta)$ that we have
already discussed, either here or previously.

\begin{lemma} If $\delta \in \Rub_U(A^2)$, then $b(\delta) \in PB_{4k}$
acting on $\hR_{2k}$ preserves all of the sets
\[ \{(z,z)\} \subsetneq A^2 \subsetneq \hU \cdot (A^2) \subsetneq
    (\hU \cdot A)^2 \subsetneq (\hR_k^0)^2 \subsetneq \hR^0_{2k}
    \subsetneq \hR_{2k}. \]
\label{l:preserve}
\end{lemma}

Note that it is easy to confuse the set $A^2$ with the slightly larger
$\hU \cdot (A^2)$ and $(\hU \cdot A)^2$, and the set $(\hR_k^0)^2$ with
the slightly larger $\hR^0_{2k}$.  In the proof, it will be crucial that
each $b(\delta)$ preserves both $A^2$ and $(\hR_k^0)^2$.

\begin{figure*}[t]
\begin{tikzpicture}[decoration={markings,
    mark=at position 0.5 with {\arrow{angle 90}}}]
\coordinate (*) at (0,-2.5);
\foreach \i in {-5.4,-4.8,-3.6,-2.4,-1.8,-.6,3.6,4.2,5.4} {
	\fill (\i,0) circle[radius=0.06]; }
\foreach \i in {-4.2,-1.2,4.8} {
	\draw (\i,0) node {$\cdots$}; }
\draw[darkgreen,postaction={decorate}] (*) .. controls (-5.4+.25,-1.6) and
    (-5.4+.25,-.15) .. (-5.4+.25,0) .. controls (-5.4+.25,.15) and
    (-5.4+.15,.25) .. (-5.4,.25) .. controls (-5.4-.15,.25) and
    (-5.4-.25,.15) .. (-5.4-.25,0) .. controls (-5.4-.25,-.15) and
    (-5.4-.25,-1.5) .. (*);
\draw[darkgreen,postaction={decorate}] (*) .. controls (-2.4+.25,-.75) and
    (-2.4+.25,-.15) .. (-2.4+.25,0) .. controls (-2.4+.25,.15) and
    (-2.4+.15,.25) .. (-2.4,.25) .. controls (-2.4-.15,.25) and
    (-2.4-.25,.15) .. (-2.4-.25,0) .. controls (-2.4-.25,-.2) and
    (-2.4,-.5) .. (*);
\draw[darkgreen,postaction={decorate}] (*) .. controls (3.6+.25,-.5) and
    (3.6+.25,-.1) .. (3.6+.25,0) .. controls (3.6+.25,.15) and
    (3.6+.15,.25) .. (3.6,.25) .. controls (3.6-.15,.25) and
    (3.6-.25,.15) .. (3.6-.25,0) .. controls (3.6-.25,-.2) and
    (3.6-.25,-.7) .. (*);
\draw[darkgreen,postaction={decorate}] (*) .. controls (-4.8-.25,-1) and
    (-4.8-.25,-.1) .. (-4.8-.25,0) .. controls (-4.8-.25,.15) and
    (-4.8-.15,.25) .. (-4.8,.25) .. controls (-4.8+.15,.25) and
    (-4.8+.25,.15) .. (-4.8+.25,0) .. controls (-4.8+.25,-.5) and
    (-4.8+.25,-.85) .. (*);
\draw[darkgreen,postaction={decorate}] (*) .. controls (-3.6-.25,-.5) and
    (-3.6-.25,-.1) .. (-3.6-.25,0) .. controls (-3.6-.25,.15) and
    (-3.6-.15,.25) .. (-3.6,.25) .. controls (-3.6+.15,.25) and
    (-3.6+.25,.15) .. (-3.6+.25,0) .. controls (-3.6+.25,-.2) and
    (-3.6+.3,-.7) .. (*);
\draw[darkgreen,postaction={decorate}] (*) .. controls (-1.8-.25,-.5) and
    (-1.8-.25,-.1) .. (-1.8-.25,0) .. controls (-1.8-.25,.15) and
    (-1.8-.15,.25) .. (-1.8,.25) .. controls (-1.8+.15,.25) and
    (-1.8+.25,.15) .. (-1.8+.25,0) .. controls (-1.8+.25,-.8) and
    (-1.8+.25,-.9) .. (*);
\draw[darkgreen,postaction={decorate}] (*) .. controls (-.6-.25,-.5) and
    (-.6-.25,-.1) .. (-.6-.25,0) .. controls (-.6-.25,.15) and
    (-.6-.15,.25) .. (-.6,.25) .. controls (-.6+.15,.25) and
    (-.6+.25,.15) .. (-.6+.25,0) .. controls (-.6+.25,-1) and
    (-.6+.25,-1.5) .. (*);
\draw[darkgreen,postaction={decorate}] (*) .. controls (4.2-.1,-1) and
    (4.2-.25,-.15) .. (4.2-.25,0) .. controls (4.2-.25,.15) and
    (4.2-.15,.25) .. (4.2,.25) .. controls (4.2+.15,.25) and
    (4.2+.25,.15) .. (4.2+.25,0) .. controls (4.2+.25,-.5) and
    (4.2+.25,-.75) .. (*);
\draw[darkgreen,postaction={decorate}] (*) .. controls (5.4-.25,-1) and
    (5.4-.25,-.15) .. (5.4-.25,0) .. controls (5.4-.25,.15) and
    (5.4-.15,.25) .. (5.4,.25) .. controls (5.4+.15,.25) and
    (5.4+.25,.15) .. (5.4+.25,0) .. controls (5.4+.25,-.15) and
    (5.4+.25,-1) .. (*);
\draw[thick,darkblue] (0,0) circle [x radius=6, y radius=2.5];
\fill[darkblue] (*) circle[radius=0.06];
\draw[darkblue] (*) node[below] {$*$};
\draw[semithick,darkblue] (*) -- (0,2.5);
\draw[semithick,darkblue] (*) -- (-3,0) -- (-3,2.1651);
\draw[semithick,darkblue] (*) -- (3,0) -- (3,2.1651);
\draw (1.5,.5) node {\Large $\cdots$};
\draw (-5.4,.2) node[above] {$\gamma_{1,1}$};
\draw (-4.8,.2) node[above] {$\gamma_{2,1}$};
\draw (-3.6,.2) node[above] {$\gamma_{2k,1}$};
\draw (-2.4,.2) node[above] {$\gamma_{1,2}$};
\draw (-1.8,.2) node[above] {$\gamma_{2,2}$};
\draw (-.6,.2) node[above] {$\gamma_{2k,2}$};
\draw (3.6,.2) node[above] {$\gamma_{1,n}$};
\draw (4.2,.2) node[above] {$\gamma_{2,n}$};
\draw (5.4,.2) node[above] {$\gamma_{2k,n}$};
\draw (-4.1,1.1) node {$(D^2\setminus[2k])_1$};
\draw (-1.5,1.5) node {$(D^2\setminus[2k])_2$};
\draw (4.1,1.1) node {$(D^2\setminus[2k])_n$};
\draw (5.4,2) node {\large $D^2 \setminus[2kn]$};
\end{tikzpicture}
\caption{Punctured disks that encode $n$ symbols of a $\ZSAT$ circuit.}
\label{f:disks} \end{figure*}
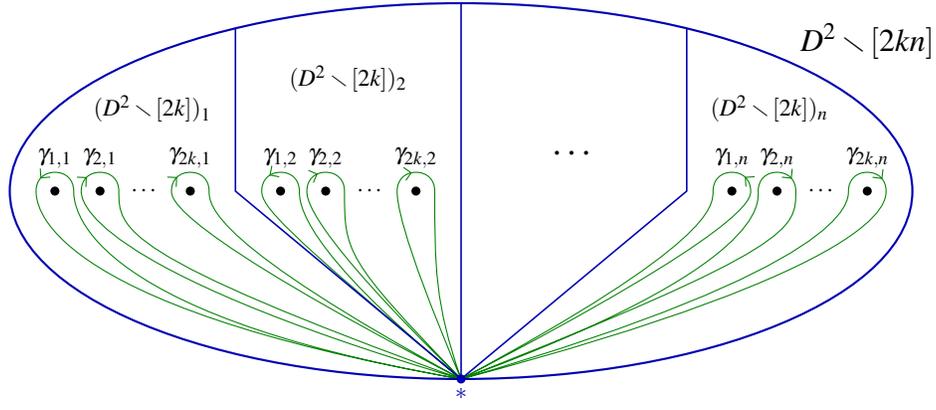

\subsection{The reduction}
\label{ss:reduction}

Let $Z$ be an instance of $\ZSAT_{U,A,I,F}$, with $U, A, I, F$ as in
\Sec{ss:alphagadget}.  Recall this means $Z$ is a planar $U$-equivariant
reversible circuit over the alphabet $A$.  Suppose that $Z$ has width $n$, so
that it acts on $n$ symbols; and length $\ell$, so that it has $\ell$ gates.

Consider the disk $D^2 \setminus [2kn]$ with $2kn$ punctures and a separate
base point $* \in \partial D^2$.  Divide it into $n$ disks $(D^2 \setminus
[2k])_i$ with $1 \le i \le n$ so that each one contains the base point,
as indicated in \Fig{f:disks}.  Also pick generators $\{\gamma_{j,i}\}$
for each $\pi_1((D^2 \setminus [2k])_i)$ as indicated in the figure, where
$1 \le j \le 2k$.

As in \Fig{f:reduction}, we convert $Z$ into a braid diagram $b(Z)$ by
replacing each strand in $Z$ with $2k$ parallel strands and each gate
$\delta^{(m)}$ in $Z$ with the braid gadget $b(\delta^{(m)})$, where here
$1 \le m \le \ell$.  Let $K(Z)$ be the oriented link diagram formed by the
plat closure of $b(Z)$ indicated in the figure, and for each $m$ with $0
\le m \le \ell$, let $(D^2 \setminus [2kn])^{(m)}$ be a disk transverse to
the braid, so that these disks and the braid gadgets alternate.  Each disk
$(D^2 \setminus [2kn])^{(m)}$ is also divided into subdisks $\{(D^2 \setminus
[2kn])^{(m)}_i\}$ as before, with loops $\{\gamma^{(m)}_{j,i}\}$.  Finally
let $\gamma_0 \in \pi_1(S^3\setminus K(Z))$ be the indicated meridian.
In other words, $\gamma_0 = \gamma^{(0)}_{1,1}$.

We are interested in homomorphisms
\[ f:\pi_1(S^3 \setminus K(Z)) \to G \]
such that $f(\gamma_0) = c$.  Using the system of disks and loops just
defined, we can restrict $f$ to other maps and elements as follows:
\begin{align*}
f^{(m)}&:\pi_1((D^2 \setminus [2kn])^{(m)}) \to G \\
f^{(m)}_i&:\pi_1((D^2 \setminus [2k])^{(m)}_i) \to G \\
f^{(m)}_{j,i}&= f(\gamma^{(m)}_{j,i}) \in C.
\end{align*}
We can also write $f^{(m)}_i \in T_k \cong C^{2k}$, and we can think of
the map $f^{(m)}_i$ as a list of the group elements $(f^{(m)}_{j,i})_{j=1}^k$.
For simplicity we rename the first and last levels of $f$:
\[ p \defeq f^{(0)} \qquad q \defeq f^{(\ell)}. \]
The inclusion map
\[ \imath_*:\pi_1((D^2 \setminus [2kn])^{(0)}) \to \pi_1(S^3 \setminus K(Z)) \]
is always a surjection and never a bijection.   Our goal is to show that a
map $p$ from the former extends to a map $f$ from the latter if and only
if $p$ corresponds to a solution to the circuit $Z$ with $q = Z(p)$.
(Moreover, that there are no non-trivial solutions if we replace $G$
with a group $J$ generated by a smaller conjugacy class.)

\begin{lemma} Let $Z$ be an instance of $\ZSAT_{U,A,I,F}$ and let $\#Z$
denote the number of solutions to $Z$.  Then the diagram $K(Z)$ and meridian
$\gamma_0$ have the following properties:
\begin{enumerate}
\item $K(Z)$ is a knot.
\item If $J$ is a non-cyclic group generated by a conjugacy class $E$
with $\#E < \#C$, then $\#Q(K(Z);J,E) = 0$.
\item $\#H(K(Z),\gamma_0;G,c) = \#Z$.
\end{enumerate}
\end{lemma}

\begin{proof} Part 1: Every braid gadget $b(\delta)$ as in \Cor{c:gadget}
is a pure braid, so our choice of plats in \Fig{f:reduction} guarantees
that $K(Z)$ is a knot, rather than a link with more than one component.

Part 2: Let $J$ be a group generated by a conjugacy class $E$ such that
$\#E < \#C$, and retain the notation $f$, $p$, and $q$ defined above for
the group $G$.  Following \Cor{c:gadget}, an arbitrary gadget $b(\delta)$
acts on $T_{2k}(J,E)$ by definition, and acts trivially by construction.
In particular each braid gadget $b(\delta^{(m)})$ acts on some pair
$(f^{(m-1)}_i,f^{(m-1)}_{i+1})$, and does nothing to that pair.  Thus for
the purpose of computing either $\#H(K(Z);J,E)$ or $\#Q(K(Z);J,E)$, $K(Z)$
is equivalent to the unknot.  Since by hypothesis $J$ is not cyclic,
we obtain $\#Q(K(Z);J,E) = 0$, as desired.

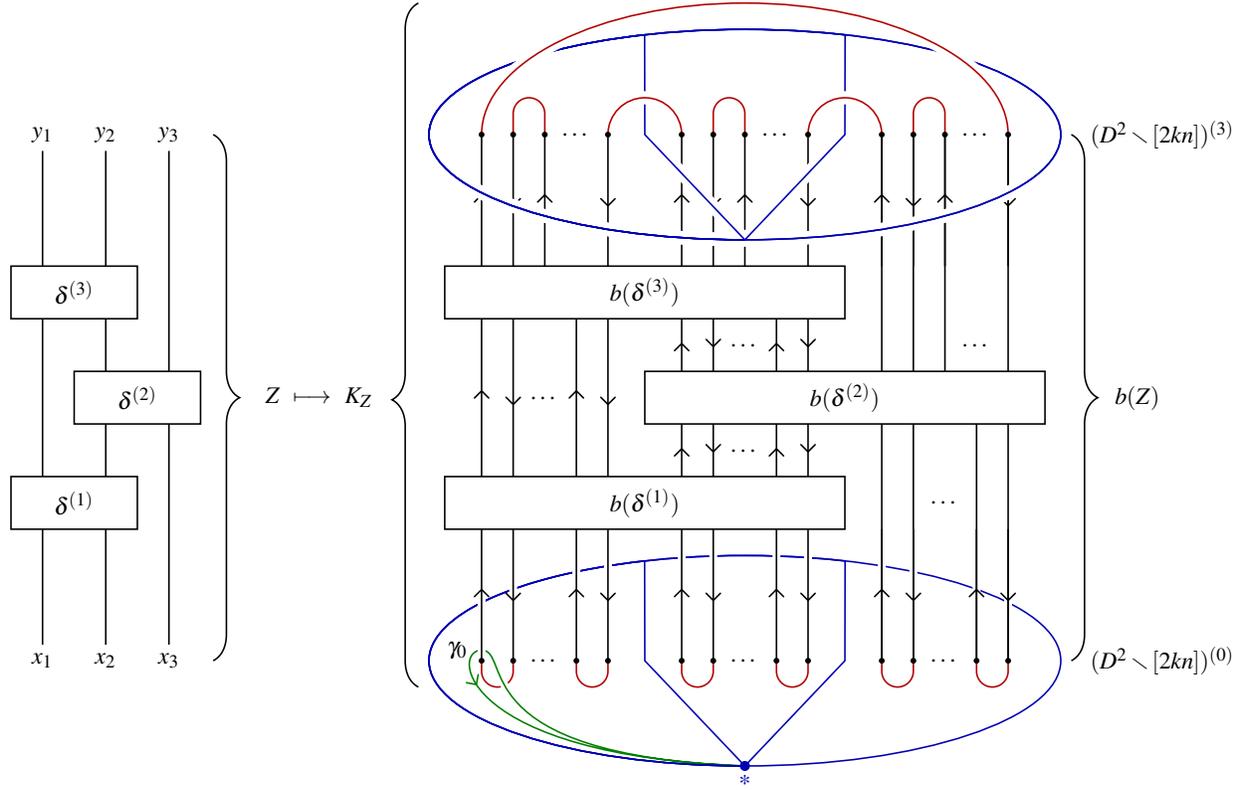
\begin{figure*}[t]
\begin{tikzpicture}[scale = 0.14,semithick,decoration={markings,
    mark=at position 0.55 with {\arrow{angle 90}}}]

\coordinate (*) at (0,0);

\draw[darkred] (3,10) -- (3,9) arc (180:360:1.5) -- (6,10);
\draw[darkred] (-3,10) -- (-3,9) arc (0:-180:1.5) -- (-6,10);
\draw[darkred] (3*7+1,10) -- (3*7+1,9) arc (180:360:1.5) -- (3*8+1,10);
\draw[darkred] (3*4+1,10) -- (3*4+1,9) arc (180:360:1.5) -- (3*5+1,10);
\draw[darkred] (-3*7-1,10) -- (-3*7-1,9) arc (0:-180:1.5) -- (-3*8-1,10);
\draw[darkred] (-3*4-1,10) -- (-3*4-1,9) arc (0:-180:1.5) -- (-3*5-1,10);

\draw[white,line width = 3] (*) .. controls (-27,1) and (-3*8+2,10+1)
    .. (-3*8-1,10+1) .. controls (-3*8-4,10+1) and (-27,1) .. (*);
\draw[darkgreen,postaction={decorate}] (*) .. controls (-27,1)
    and (-3*8+2,10+1) .. (-3*8-1,10+1) .. controls (-3*8-4,10+1) and
    (-27,1) .. (*);

\draw[darkblue] (*) arc (270:-90:30 and 10);
\draw[darkblue] (*) arc (270:108.5:30 and 10) -- (-9.5,10) -- (*);
\draw[darkblue] (*) arc (270:71.5:30 and 10) -- (9.5,10) -- (*);
\node[right] at (32,10) {$(D^2 \setminus [2kn])^{(0)}$};

\draw[fill,darkblue] (*) circle [radius = 0.4];
\node[below,darkblue] at (*) {*};

\foreach \i in {-2,1} {
	\coordinate (p) at (3*\i,10);
	\draw[white, line width = 3] (p) -- ($(p) + (0,12.5)$);
	\draw[postaction={decorate}] (p) -- ($(p) + (0,12.5)$);
	\draw[postaction={decorate}] ($(p) + (0,17.5)$) -- ($(p) + (0,22.5)$);
	\draw[postaction={decorate}] ($(p) + (0,27.5)$) -- ($(p) + (0,32.5)$);
	\draw[fill] (p) circle [radius = 0.2]; }

\foreach \i in {-1,2} {
	\coordinate (p) at (3*\i,10);
	\draw[white, line width = 3] (p) -- ($(p) + (0,12.5)$);
	\draw[postaction={decorate}] ($(p) + (0,12.5)$) -- (p);
	\draw[postaction={decorate}] ($(p) + (0,22.5)$) -- ($(p)+(0,17.5)$);
	\draw[postaction={decorate}] ($(p) + (0,32.5)$) -- ($(p)+(0,27.5)$);
	\draw[fill] (p) circle [radius = 0.2]; }

\foreach \i in {5,8} {
	\coordinate (q) at (-3*\i-1,10);
	\draw[white,line width = 3] (q) -- ($(q) + (0,12.5)$);
	\draw[postaction={decorate}] (q) -- ($(q) + (0,12.5)$);
	\draw[postaction={decorate}] ($(q)+(0,17.5)$) -- ($(q) + (0,32.5)$);
	\draw[fill] (q) circle [radius = 0.2]; }
\foreach \i in {4,7} {
	\coordinate (q) at (-3*\i-1,10);
	\draw[white,line width = 3] (q) -- ($(q) + (0,12.5)$);
	\draw[postaction={decorate}] ($(q) + (0,12.5)$)--(q);
	\draw[postaction={decorate}] ($(q) + (0,32.5)$) -- ($(q)+(0,17.5)$);
	\draw[fill] (q) circle [radius = 0.2]; }

\foreach \i in {4,7} {
	\coordinate (p) at (3*\i+1,10);
	\draw[white,line width = 3] (p) -- ($(p) + (0,35)$);
	\draw[postaction={decorate}] (p) -- ($(p) + (0,12.5)$);
	\draw (p) -- ($(p) + (0,22.5)$);
	\draw[fill] (p) circle [radius = 0.2]; }
\foreach \i in {5,8} {
	\coordinate (q) at (3*\i+1,10);
	\draw[white,line width = 3] (q) -- ($(q) + (0,12.5)$);
	\draw[postaction={decorate}] ($(q) + (0,12.5)$) --(q);
	\draw (q) -- ($(q) + (0,22.5)$);
	\draw[fill] (q) circle [radius = 0.2]; }

\node at (0,10) {  \dots};
\node at (3*6+1,10) { \dots};
\node at (-3*6-1,10) { \dots};
\node[left] at (-3*8-1.5,10+1) { $\gamma_0$};
\node at (-3*6-1,10+12.5+5+5+2.5) { \dots};
\node at (3*6+1,12.5+5+5+2.5) { \dots};
\node at (0,12.5+5+5+2.5+5) { \dots};
\node at (0,12.5+5+5+2.5+5+10) { \dots};

\foreach \i in {0,-2} {
	\coordinate (q) at (3*\i,60);
	\draw[postaction={decorate}] ($(q) - (0,12.5)$) -- (q);	}
\foreach \i in {-1,2} {
	\coordinate (q) at (3*\i,60);
	\draw[postaction={decorate}] (q) -- ($(q) - (0,12.5)$);	}
\node at (3,60) { \dots};
\foreach \i in {4,5,7,8} {
	\coordinate (q) at (3*\i+1,60);
	\draw[white, line width = 2] (q) -- ($(q) - (0,12.5+10)$); }
\foreach \i in {4,5,6,8} {
	\coordinate (q) at (3*\i+1,60);
	\draw (q) -- ($(q) - (0,12.5+10)$); }
\foreach \i in {5,8} {
	\coordinate (q) at (3*\i+1,60);
	\draw[postaction={decorate}] (q) -- ($(q) - (0,12.5)$); }
\foreach \i in {4,6} {
	\coordinate (q) at (3*\i+1,60);
	\draw[postaction={decorate}] ($(q) - (0,12.5)$) -- (q); }
\foreach \i in {4,7} {
	\coordinate (p) at (-3*\i-1,60);
	\draw[postaction={decorate}] (p) -- ($(p) - (0,12.5)$); }
\foreach \i in {6,8} {
	\coordinate (p) at (-3*\i-1,60);
	\draw[postaction={decorate}] ($(p) - (0,12.5)$) -- (p); }

\node at (3*7+1,40) { \dots};
\node at (3*7+1,60) { \dots};
\node at (-3*5-1,60) { \dots};

\draw[fill=white] (-28.5,22.5) rectangle (9.5,27.5);
\node at ($ (-28.5,22.5)!0.5!(9.5,27.5) $) { $b(\delta^{(1)})$};
\draw[fill=white] (28.5,32.5) rectangle (-9.5,37.5);
\node at ($ (28.5,32.5)!0.5!(-9.5,37.5) $) { $b(\delta^{(2)})$};
\draw[fill=white] (-28.5,42.5) rectangle (9.5,47.5);
\node at ($ (-28.5,42.5)!0.5!(9.5,47.5) $) { $b(\delta^{(3)})$};

\draw [decorate,decoration={brace,amplitude=10pt}] (31,60) -- (31,10)
    node [midway,xshift=25] {$b(Z)$};
\draw[mybrace=0.42] (-31,7.5) -- (-31,72.5) ;

\begin{scope}[xscale=0.6,xshift=-1000]
\draw [decorate,decoration={brace,amplitude=10pt}] (-49,60) -- (-49,10)
    node [midway,xshift=40] { $Z \hspace{3pt} \longmapsto \hspace{3pt} K_Z$};
\node (x3) at (-56,10) {$x_3$}; \node (x2) at (-66,10) {$x_2$};
\node (x1) at (-76,10) {$x_1$}; \node (y3) at (-56,60) {$y_3$};
\node (y2) at (-66,60) {$y_2$}; \node (y1) at (-76,60) {$y_1$};
\draw (x1) -- (y1); \draw (x2) -- (y2); \draw (x3) -- (y3);
\draw[fill=white] ($(x1)+(-5,12.5)$) rectangle ($(x2)+(5,17.5)$);
\node at ($(-81,22.5)!0.5!(-61,27.5)$) { $\delta^{(1)}$};
\draw[fill=white] ($(x2)+(-5,22.5)$) rectangle ($(x3)+(5,27.5)$);
\node at ($ (-71,32.5)!0.5!(-51,37.5) $) { $\delta^{(2)}$};
\draw[fill=white] ($(x1)+(-5,32.5)$) rectangle ($(x2)+(5,37.5)$);
\node at ($(-81,42.5)!0.5!(-61,47.5)$) { $\delta^{(3)}$};
\end{scope}

\begin{scope}[yshift=50cm]
\draw[white,line width=5] (0,0) arc (270:-90:30 and 10);
\draw[white,line width=5] (0,0) arc (270:-90:30 and 10);
\draw[white,line width=5] (0,0) arc (270:108.5:30 and 10) -- (-9.5,10) -- (0,0);
\draw[white,line width=5] (0,0) arc (270:71.5:30 and 10) -- (9.5,10) -- (0,0);
\draw[darkblue] (0,0) arc (270:-90:30 and 10);
\draw[darkblue] (0,0) arc (270:-90:30 and 10);
\draw[darkblue] (0,0) arc (270:108.5:30 and 10) -- (-9.5,10) -- (0,0);
\draw[darkblue] (0,0) arc (270:71.5:30 and 10) -- (9.5,10) -- (0,0);
\node[right] at (32,10) {$(D^2 \setminus [2kn])^{(3)}$};
\end{scope}

\draw[white,line width=5] (-25,60) arc (180:0:25 and 12.5);
\draw[white,line width=5] (-22,60) -- (-22,62) arc (180:0:1.5) -- (-19,60);
\draw[white,line width=5] (-13,60) arc (180:0:3.5);
\draw[white,line width=5] (-3,60) -- (-3,62) arc (180:0:1.5) -- (0,60);
\draw[white,line width=5] (13,60) arc (0:180:3.5);
\draw[white,line width=5] (16,60) -- (16,62) arc (180:0:1.5) -- (19,60);

\draw[darkred] (-25,60) arc (180:0:25 and 12.5);
\draw[darkred] (-22,60) -- (-22,62) arc (180:0:1.5) -- (-19,60);
\draw[darkred] (-13,60) arc (180:0:3.5);
\draw[darkred] (-3,60) -- (-3,62) arc (180:0:1.5) -- (0,60);
\draw[darkred] (13,60) arc (0:180:3.5);
\draw[darkred] (16,60) -- (16,62) arc (180:0:1.5) -- (19,60);

\draw[fill] (-25,60) circle [radius=.2];
\draw[fill] (-22,60) circle [radius=.2];
\draw[fill] (-19,60) circle [radius=.2];
\draw[fill] (-13,60) circle [radius=.2];
\draw[fill] (-6,60) circle [radius=.2];
\draw[fill] (-3,60) circle [radius=.2];
\draw[fill] (0,60) circle [radius=.2];
\draw[fill] (6,60) circle [radius=.2];
\draw[fill] (13,60) circle [radius=.2];
\draw[fill] (16,60) circle [radius=.2];
\draw[fill] (19,60) circle [radius=.2];
\draw[fill] (25,60) circle [radius=.2];
\end{tikzpicture}
\caption{Reducing a circuit $Z$ with $n=3$ variables to the knot $K_Z$.}
\label{f:reduction} \end{figure*}

Part 3: Let $X(Z)$ be the set of solutions to the circuit $Z$.  We will
show that $X(Z) = H(K(Z),\gamma_0;G,c)$ in the natural sense.  If $q =
Z(p)$ is a solution to $Z$, then by definition,
\[ (p_1,p_2,\dots,p_n) \in I^n \qquad Z(p) = (q_1,q_2,\dots,q_n) \in F^n. \]
By \Cor{c:gadget}, each braid gadget $b(\delta)$ acts on $A^2$ exactly as
$\delta$ does, and therefore the braid $b(Z)$ acts on $A^n$ exactly as
the circuit $Z$ does.  By the definition of the initial subalphabet $I$,
the map $p = f^{(0)}$ factors through the plat attached to the bottom of
the braid $b(Z)$.  Meanwhile, the definition of the alphabet $A$ together
with the definition of the final subalphabet $F$ together imply that $q
= b(Z) \cdot p$ factors through the plat attached to the top of $b(Z)$.
Most of the U-turns at the top of the plat are internal to one symbol $q_i
\in A$, and these force $q_i \in F$.  The others connect either $q_{2k,i}$
with $q_{1,i+1}$ or $q_{2k,n}$ with $q_{1,1}$.  These constraints hold
automatically in the alphabet $A$, because they reduce to the equation $c=c$.
Finally $p_1 \in A$ also gives us that $f(\gamma_0) = c$.  This establishes
that $X(Z) \subseteq H(K(Z),\gamma_0;G,c)$.  In fact it establishes
a little more, namely that any other element of $H(K(Z),\gamma_0;G,c)$
cannot come from $p = f^{(0)} \in A^n$.

Conversely, let $f \in H(K(Z),\gamma_0;G,c) \setminus X(Z)$ be a hypothetical
spurious homomorphism.   Then tautologically $p = f^{(0)} \in T_k^n \cong
(C^{2k})^n$, but we quickly obtain an important restriction.  Each $p_i$
factors through the initial plat attached to $(D^2 \setminus [2k])^{(0)}_i$,
so \Lem{l:schprops} tells us that $\sch(p_i) = 0$ and thus that $p_i \in
\hR_k^0$.  Moreover, \Lem{l:preserve} tells us that every braid gadget
preserves this condition, so $f^{(m)}_i \in \hR_k^0$ for every $m$ and $i$.
In other words, we can interpret $b(Z)$ as a circuit that uses the larger
alphabet $\hR_k^0 \supseteq A$.

We claim that we can further restrict the alphabet to $\hU \cdot A$.
If $p_i = f^{(0)}_i \in \R^0_k \setminus \hU \cdot A$ for some $i$, then
\Cor{c:gadget} also tells us that no gate gadget changes this value,
so that in particular $q_i = p_i$.   But then the initial and final
plat closures together tell us that
\[ p_i = (e,e,\dots,e) \]
for some $e \in C$, which thus means that
\[ p_i = q_i \in \Inn(G) \cdot \{z\} \subseteq \hU \cdot A \]
after all.  By \Lem{l:preserve}, the condition that $p = f^{(0)} \in
(\hU \cdot A)^n$ is also preserved through every gate gadget in $b(Z)$.

We now show that $f^{(m)} \in \hU \cdot A^n$ for every $m$.  The condition
that $f^{(m)}_i \in \hU \cdot A$ tells us that each symbol $f^{(m)}_i$
begins and ends with the same group element $e \in C$, and what we would
like to know is that $e$ does not depend on $i$.   The final plat closure
makes this immediate for $q = f^{(\ell)}$, and then \Lem{l:preserve} tells
us that the condition is preserved in reverse $f^{(m)}$ as $m$ decreases.

Finally, because $p = f^{(0)} \in \hU \cdot A^n$ and $f(\gamma_0) = c$,
we conclude that $p \in A^n$.
\end{proof}

To conclude the proof of \Thm{th:main}, the knot $K(Z)$ can be constructed
from $Z$ in polynomial time as a function of the number of gates in
$Z$, since it is just direct replacement of each gate $\delta$ by
the corresponding gate gadget $b(\delta)$.  Thus it is a parsimonious
reduction from $\#\ZSAT_{U,A,I,F}$ to $\#H(-,G,c)$ that preserves the
$\shP$-completeness properties state in \Thm{th:zsat}.

\begin{remark} As in our previous work \cite{K:zombies}, the
proof of \Thm{th:main} establishes an efficient bijection between
$Q(K(Z),\gamma_0;G,c)$ and the orbits of non-trivial solutions to
$\#\ZSAT_{U,A,I,F}$, and therefore the set of certificates in any problem
in $\shP$.  This is an even stronger property than parsimonious reduction
known as Levin reduction.
\end{remark}

\section{Open problems}

As with our previous theorem about homology 3-spheres \cite{K:zombies},
we conjecture that $\#Q(K;G,C)$ is also computationally intractable when
$K$ is a randomly chosen knot.  There are various inequivalent models for
choosing a knot at random \cite{Even:models}, and we believe that $\#Q(K;G,C)$
should be intractable for many of them.  Hardness in random cases is a
known property for some $\shP$-complete problems \cite{CPS:permanent}.

Also in our previous work, we first had in mind that the analogous invariant
$\#Q(M;G)$ is intractable for 3-manifolds $M$; later we sharpened the
construction to make $M$ a homology 3-sphere.  \Thm{th:main} is in keeping
with the analogy that a homology 3-sphere is like a knot, while a
general 3-manifold is like a link.   However, a deeper analogy is that a
homology 3-sphere, among 3-manifolds, is like a knot with trivial Alexander
polynomial, among knots.  We conjecture that \Thm{th:main} also holds for
knots with trivial Alexander polynomial.  This would better motivate the
restriction that $G$ should be a non-abelian simple group.

% \bibliography{co,gt,nt,qa,qp,me,web}

\providecommand{\bysame}{\leavevmode\hbox to3em{\hrulefill}\thinspace}
\providecommand{\href}[2]{#2}

\typeout{get arXiv to do 4 passes: Label(s) may have changed. Rerun}

\end{document}